\documentclass[UTF-8,reqno]{amsart}
\usepackage{enumerate}
\usepackage{amssymb,url,color, booktabs}

\usepackage{mathrsfs}
\usepackage{amsmath}

\usepackage{fancyhdr}
\usepackage[nobysame]{amsrefs}
\BibSpec{article}{%
+{}{\PrintAuthors} {author}
+{,}{ \textrm} {title}
+{.}{ \textit} {journal}
+{,}{ \textbf} {volume}
+{}{ \parenthesize} {date}
+{,}{ } {pages}
+{.}{ arXiv:} {eprint}
+{.}{} {transition}
}
\BibSpec{book}{%
+{}{\PrintAuthors} {author}
+{,}{ \textit} {title}
+{.}{ \textrm} {series} %+{,}{ \textrm} {series}
+{,}{ Vol.} {volume}
+{.}{ } {publisher}
+{,}{ } {date}
+{.}{} {transition}
}
\usepackage{color}
\usepackage[colorlinks=true]{hyperref}
\hypersetup{
    linkcolor=blue,          % color of internal links
    citecolor=red,        % color of links to bibliography
    filecolor=blue,      % color of file links
    urlcolor=cyan
}

\numberwithin{equation}{section}

\newcommand{\be}{\begin{eqnarray}}
\newcommand{\ee}{\end{eqnarray}}
\newcommand{\ce}{\begin{eqnarray*}}
\newcommand{\de}{\end{eqnarray*}}
\newtheorem{theorem}{Theorem}[section]
\newtheorem{lemma}[theorem]{Lemma}
\newtheorem{remark}[theorem]{Remark}
\newtheorem{definition}[theorem]{Definition}
\newtheorem{proposition}[theorem]{Proposition}
\newtheorem{problem}[theorem]{Problem}
\newtheorem{Examples}[theorem]{Example}
\newtheorem{corollary}[theorem]{Corollary}
\newtheorem{assumption}[theorem]{Assumption}

\newenvironment{proof of theorem 1.2}{{\it Proof of Theorem 1.2}.}{{\hfill 	
$\square$\hskip - \parfillskip}}
\newenvironment{proof of theorem 1.7}{{\it Proof of Theorem 1.7}.}{{\hfill 	
		$\square$\hskip - \parfillskip}}
\newenvironment{proof of theorem 1.3}{{\it Proof of Theorem 1.3}.}{{\hfill 	
		$\square$\hskip - \parfillskip}}
\newenvironment{proof of theorem 1.5}{{\it Proof of Theorem 1.5}.}{{\hfill 	
		$\square$\hskip - \parfillskip}}
\newenvironment{proof of theorem 1.4}{{\it Proof of Theorem 1.4}.}{{\hfill 	
		$\square$\hskip - \parfillskip}}
\newenvironment{proof of theorem 1.6}{{\it Proof of Theorem 1.6}.}{{\hfill 	
		$\square$\hskip - \parfillskip}}
\newenvironment{proof of theorem 5.3}{{\it Proof of Theorem 5.3}.}{{\hfill 	
		$\square$\hskip - \parfillskip}}
\newenvironment{proof of (1.3)}{{\it Proof of (1.3)}.}{{\hfill 	
		$\square$\hskip - \parfillskip}}
\newenvironment{proof of corollary 1.8}{{\it Proof of Corollary 1.8}.}{{\hfill 	
			$\square$\hskip - \parfillskip}}
\newenvironment{proof of corollary 1.6 and 1.7}{{\it Proof of Corollary 1.6 and 1.7}.}{{\hfill 	
		$\square$\hskip - \parfillskip}}

\makeatletter
\newcommand{\rmnum}[1]{\romannumeral #1}
\newcommand{\Rmnum}[1]{\expandafter\@slowromancap\romannumeral #1@}
\makeatother

\def\eps{\varepsilon}

\def\a{\alpha}
\def\om{\omega}
\def\Om{\Omega}

\def\p{\partial}

\def\l{\lambda}
\def\la{\langle}
\def\ra{\rangle}
\def\[{{\Big[}}
\def\]{{\Big]}}
\def\<{{\langle}}
\def\>{{\rangle}}
\def\({{\Big(}}
\def\){{\Big)}}

\def\bx{{\mathbf{x}}}

\def\min{{\mathord{{\rm min}}}}

\def\={&\!\!=\!\!&}

\def\cK{{\mathcal K}}
\def\cL{{\mathcal L}}

\def\cS{{\mathcal S}}

\def\mR{{\mathbb R}}
\def\mS{{\mathbb S}}

\def\1{{\mathbf{1}}}

\def\sA{{\mathscr A}}

\def\geq{\geqslant}
\def\leq{\leqslant}
\def\ge{\geqslant}
\def\le{\leqslant}

\def\k{\kappa}

\def\eps{\varepsilon}

\def\a{\alpha}
\def\om{\omega}
\def\Om{\Omega}

\def\p{\partial}

\def\l{\lambda}
\def\la{\langle}
\def\ra{\rangle}
\def\[{{\Big[}}
\def\]{{\Big]}}
\def\<{{\langle}}
\def\>{{\rangle}}
\def\({{\Big(}}
\def\){{\Big)}}

\def\bx{{\mathbf{x}}}

\def\min{{\mathord{{\rm min}}}}

\def\={&\!\!=\!\!&}
\def\bt{\begin{theorem}}
\def\et{\end{theorem}}
\def\bl{\begin{lemma}}
\def\el{\end{lemma}}
\def\br{\begin{remark}}
\def\er{\end{remark}}
\def\bx{\begin{Examples}}
\def\ex{\end{Examples}}
\def\bd{\begin{definition}}
\def\ed{\end{definition}}
\def\bp{\begin{proposition}}
\def\ep{\end{proposition}}
\def\bc{\begin{corollary}}
\def\ec{\end{corollary}}

\def\geq{\geqslant}
\def\leq{\leqslant}
\def\ge{\geqslant}
\def\le{\leqslant}

 \def\nn{\nabla}

\def\<{\langle} \def\>{\rangle}

\def\bpf{\begin{proof}}
\def\epf{\end{proof}}

\allowdisplaybreaks

\begin{document}
	
\title{anisotropic flows without global forcing terms and dual Orlicz Christoffel-Minkowski type problem}\thanks{\it {This research was partially supported by NSFC (Nos. 11871053 and 12261105).}}
%\date{\today}
\author{Shanwei Ding*$^1$ and Guanghan Li$^1$}

%\thanks{{\it Data availability statements: Data sharing not applicable to this article as no datasets were generated or analysed during the current study.}}


\thanks{{\it *Corresponding author. E-mail: dingsw@whu.edu.cn}}

\thanks{{\it Guanghan Li: ghli@whu.edu.cn}}

\thanks{\it $^1$ School of Mathematics and Statistics, Wuhan University, Wuhan 430072, China.
}

\begin{abstract}
In this paper, we study the long-time existence and asymptotic behavior for a class of anisotropic non-homogeneous curvature flows without global forcing terms. By the stationary solutions of such anisotropic flows, we obtain existence results for a class of dual Orlicz Christoffel-Minkowski type problems, which is equivalent to solve the PDE $G(x,u_K,Du_K)F(D^2u_K+u_K\Rmnum{1})=1$ on $\mS^n$ for a convex body $K$, where $D$ is the covariant derivative with respect to the standard metric on $\mS^n$ and $\Rmnum{1}$ is the unit matrix of order $n$. This result covers many previous known solutions to $L^p$ dual Minkowski problem, $L^p$ dual  Christoffel-Minkowski problem, and some dual Orlicz Minkowski problem etc.. Meanwhile, the variational formula of some modified quermassintegrals and the corresponding prescribed area measure problem (Orlicz Christoffel-Minkowski type problem) are considered, and inequalities involving modified quermassintegrals are also derived.  As corollary, this gives a partial answer about the general prescribed curvature problem raised in \cite{GRW}.

\end{abstract}

\maketitle
\setcounter{tocdepth}{2}
\tableofcontents

\section{Introduction}
%Flows of convex hypersurfaces  by a class of speed functions which are homogenous and symmetric in principal curvatures have been extensively studied in the past four decades. Well-known examples include the mean curvature flow \cite{HG}, and the Gauss curvature flow \cite{BS,FWJ}. In \cite{HG} Huisken showed that the flow has a unique smooth solution and the hypersurface converges to a round sphere if the initial hypersurface is closed and convex. Later, a range of flows with the speed of homogenous of degree one in principal curvatures were established,  see \cite{B0,B1,CB1,CB2} and references therein.

Anisotropic inverse curvature flows of strictly convex hypersurfaces with speed depending on their curvatures, support function and radial function have been considered recently, cf. \cite{IM,DL3,BIS,CL} etc.. These flows usually provide alternative proofs and smooth category approach of the existence of solutions to elliptic PDEs arising in convex body geometry. However, the literature on non-homogeneous anisotropic inverse curvature  flows is not very rich and there are few works in this direction, cf. \cite{JLL,BIS2,BIS3}. One advantage of this method is that there is no need to employ the constant rank theorem. For example, if $F=\sigma_{k}^{\frac{1}{k}}(\l)$, by the lower bound of $F^{-1}=\sigma_{k}^{-\frac{1}{k}}(\l)=(\frac{\sigma_{n}}{\sigma_{n-k}})^\frac{1}{k}(\k)$ one has that the hypersurfaces naturally preserve convexity. Whether these flows can be extended is an interesting problem. However, if $F\ne \sigma_{n}^\frac{1}{n}(\l)$, there is no suitable monotone integral quantities to be used to prove the convergence of flows if speed depends on  $X$. More details can be seen in our previous work \cite{DL3}, and inspired  by that paper, we consider a large class of anisotropic flows without global forcing terms.

 Let $M_0$ be a closed, smooth and strictly convex hypersurface in $\mathbb{R}^{n+1}$ ($n\geq2$), and $M_0$ encloses the origin. In this paper, we study the following inverse curvature flow
   \begin{equation}
 	\label{1.2}
 	\begin{cases}
 		&\frac{\partial X}{\partial t}(x,t)=(G(\nu,X)F^{\beta}(\l_i)-1)u\nu,\\
 		&X(\cdot,0)=X_0,
 	\end{cases}
 \end{equation}
where $F$ is a suitable curvature function of the hypersurface $M_t$ parameterized by a smooth embedding $X(x,t): \mS^n\times[0,T^*)\to \mR^{n+1}$, $\l=(\l_1,\cdots,\l_n)$ are the principal curvature radii of the  hypersurface $M_t$, $\beta>0$, $G:\mS^n\times\mR^{n+1}\rightarrow(0,+\infty)$ is a smooth function if $\langle\nu,X\rangle>0
$, $u$ is the support function and $\nu(\cdot,t)$ is the outer unit normal vector field to $M_t$.

Note that the radial function $\rho$ and support function $u$ are respectively given by $\rho=\sqrt{\langle X,X\rangle}$, $u=\langle X,\nu\rangle$, then $X=Du+u\nu$, and thus $G(\nu,X)$ can be regarded as $G(\nu,u,Du)$, where $\langle\cdot,\cdot\rangle$ is the standard inner product in $\mR^{n+1}$,  $D$ is the covariant derivative with respect to an orthonormal frame on $\mS^n$. We also denote by $\bar\nn$ the covariant derivative with respect to the metric in Euclidean space, and then $X=\bar\nn u$. Flow (\ref{1.2}) can be rewritten as the following flow
  \begin{equation*}
	\begin{cases}
		&\frac{\partial X}{\partial t}(x,t)=(G(\nu,u,Du)F^{\beta}(\l_i)-1)u\nu,\\
		&X(\cdot,0)=X_0.
	\end{cases}
\end{equation*}
We will use $G=G(\nu,X)$ to make the conditions easier to understand in our theorems, and use $G=G(\nu,u,Du)$ in the proof to emphasize the dependence of the equation on independent variables  in the rest of the paper.

Flow (\ref{1.2}) is inspired by our previous work \cite{DL3}. However, flow (\ref{1.2}) is more complicated than the flow in \cite{DL3}, since it involves a nonlinear function $G$. Note that $G$ is a function of $X$ other than $\vert X\vert$, which needs more effort to deal with. To explain this flow above, we first introduce the following flow and its scaling process, i.e. we let $G=\psi(\nu)u^{\alpha-1}\rho^\delta$ and consider the expanding flow
 \begin{equation}
 	\label{1.3}
 	\begin{cases}
 		&\frac{\partial X}{\partial t}(x,t)=\psi u^\alpha\rho^\delta F^{\beta}(x,t) \nu,\\
 		&X(\cdot,0)=X_0.
 	\end{cases}
 \end{equation}
If $\a+\beta+\delta<1$, let $\widetilde X(\cdot,\tau)=\varphi^{-1}(t)X(\cdot,t)$, where
\begin{equation}\label{1.4}
	\tau=\frac{\log((1-\alpha-\delta-\beta)t+C_0)-\log C_0}{1-\alpha-\delta-\beta},
\end{equation}
\begin{equation}\label{1.5}
\varphi(t)=(C_0+(1-\beta-\delta-\alpha)t)^{\frac{1}{1-\beta-\delta-\alpha}}.
\end{equation}
We can set $C_0$ sufficiently large to make
\begin{equation}\label{1.6}
\psi\widetilde u^{\alpha-1}\widetilde\rho^\delta\widetilde F^{\beta}\vert_{M_0}=C_0\psi u^{\alpha-1}\rho^\delta F^{\beta}\vert_{M_0}>1.
\end{equation}
Obviously, $\frac{\p\tau}{\p t}=\varphi^{\a+\delta+\beta-1}$ and $\tau(0)=0$. Then $\widetilde X(\cdot,\tau)$ satisfies the following normalized flow,
\begin{equation}\label{1.7}
	\begin{cases}
		\frac{\partial \widetilde X}{\p \tau}(x,\tau)=\psi\widetilde u^\alpha\widetilde\rho^\delta\widetilde F^{\beta}(\widetilde\l)\nu-\widetilde X,\\[3pt]
		\widetilde X(\cdot,0)=\widetilde X_0.
	\end{cases}
\end{equation}
For convenience we still use $t$ instead of $\tau$ to denote the new time variable, and omit the ``tilde'' if no confusions arise and we mention the scaled flow or normalized flow. We can find that the flow (\ref{1.7}) is equivalent (up to an isomorphism) to
\begin{equation}\label{1.8}
	\begin{cases}
		&\frac{\partial X}{\partial t}=(\psi u^\alpha\rho^\delta F^{\beta}(\l)-u) \nu,\\
		&X(\cdot,0)=X_0.
	\end{cases}
\end{equation}
This is why we consider flow (\ref{1.2}).

We can obtain convergence results for a large class of speeds and therefore make the following assumption.
\begin{assumption}\label{a1.1}
	Let $\Gamma_+=\{(\l_i)\in\mathbb{R}^n:\l_i>0\}$ be a symmetric, convex, open cone,
and suppose that $F$ is positive in $\Gamma_+$, homogeneous of degree $1$, and inverse concave with
\begin{equation*}
	\frac{\p F}{\p\l_i}>0,\quad F_*\vert_{\p\Gamma_+}=0,\quad F^{\beta}(1,\cdots,1)=1.
\end{equation*}
\end{assumption}
We remark that the function $F$ is inverse concave means the dual function $F_*$ defined by $F_*(\k_1,\cdots,\k_n)=\frac{1}{F(1/\k_1,\cdots,1/\k_n)}$ is concave on $\Gamma_+$.

We also note that $F(\l_1,\cdots,\l_n)$ can be regarded as $F([a_{ij}])$, where $\l_1,\cdots,\l_n$ are the eigenvalues of matrix $[a_{ij}]$. It is not difficult to see that the eigenvalues of matrix $[F^{ij}]=[\frac{\p F}{\p a_{ij}}]$ are $\frac{\p F}{\p\l_1},\cdots,\frac{\p F}{\p\l_n}$.

We first prove the long time existence and convergence of the above flows.
\begin{theorem}\label{t1.2}
Let $F\in C^2(\Gamma_+)\cap C^0(\p\Gamma_+)$ satisfy Assumption \ref{a1.1}, and let $M_0$ be a closed, smooth, uniformly convex hypersurface in $\mathbb{R}^{n+1}$, $n\ge2$, enclosing the origin. Suppose

$(i)$ $\lim\sup_{s\rightarrow+\infty}[\max_{y=sx}G(x,y)s^\beta]<1<\lim\inf_{s\rightarrow0^+}[\min_{y=sx}G(x,y)s^\beta]$,

$(ii)$ $(D_{x_i}D_{x_j}(Gu)^\frac{1}{\beta+1}+(Gu)^\frac{1}{\beta+1}\delta_{ij})>0$, where $x_i$ is the
orthonormal frame on $\mS^n$ and $D_{x_i}G$ is the covariant derivative of $G(\nu,X)$ with respect to the first independent variable for fixed $X$.

\noindent Then flow (\ref{1.2}) with $\beta>0$ has a unique smooth strictly convex solution $M_t$ for all time $t>0$.

 In addition, if

$(iii)$  we choose the initial hypersurface $M_0$ satisfying
$G(\nu,X)F^\beta\vert_{M_0}\ge1$ or $G(\nu,X)F^\beta\vert_{M_0}\le1$.

\noindent Then a subsequence of $M_t$ converges in $C^\infty$-topology to a positive, smooth, uniformly convex solution to $G(\nu,X)F^\beta=1$ along flow (\ref{1.2}) with $\beta>0$.

In particular, if $F=\sigma_{n}(\l)^\frac{1}{n}$, the condition (ii) can be dropped.
\end{theorem}

\textbf{Remark:}
(1) Condition (i), a weaker condition than that in \cite{LL,BIS3}, is used to prove the $C^0$ estimate. Condition (iii) is preserved along flow (\ref{1.2}), and can be used to derive the convergence. Condition (ii) is used to derive the $C^2$ estimate. In particular, Gauss curvature has pretty good properties, and therefore condition (ii) is redundant if $F=\sigma_n^{\frac1n}$.
There are lots of  hypersurfaces satisfying condition (iii). In fact, spheres with radius large enough or small enough  satisfy condition (iii) by condition (i).

(2) Condition (ii) is necessary for curvature functions except $\sigma_{n}^\frac{1}{n}$. If $G=\psi(\nu)\varphi(u)$ and $F^\beta=\sigma_{k}$, we  derive $D_{x_i}D_{x_i}(\psi\varphi(u)u)^\frac{1}{k+1}+(\psi\varphi(u)u)^\frac{1}{k+1}>0$. One can find that condition (ii) is equivalent to that in \cite{JLL} in this case.  If $G=\psi(\nu)u^{\a-1}$ and $F=\sigma_{k}^\frac{1}{k}$  for $1\le k<n$, $\a<k+1$ and $\beta=k$, Ivaki in \cite{IM} showed the existence of a rotationally symmetric $\psi$ with $((\psi^\frac{1}{k+1-\a})_{\theta\theta}+\psi^\frac{1}{k+1-\a})|_{\theta=0}<0$ and a smooth, closed, strictly convex initial hypersurface for which the solution to flow (\ref{1.2}) with $k<n$ will lose smoothness. Therefore $(D_iD_j\psi^{\frac{1}{1+\beta-\a}}+\delta_{ij}\psi^{\frac{1}{1+\beta-\a}})\ge0$ is essential to ensure the smoothness of the solution is preserved, and thus we believe that condition (ii) is necessary. Up to now, all results by use of anisotropic flows to solve Christoffel-Minkowski type problem require the condition $(D_iD_j\psi^{\frac{1}{1+\beta-\a}}+\delta_{ij}\psi^{\frac{1}{1+\beta-\a}})>0$. More details can be seen in \cite{BIS2,IM}.

(3)  We mention that condition (ii) is a structural condition for $G$ and it is independent of initial hypersurface and time.
 For example, if $G=\psi(\nu)u^{\a-1}\rho^\delta$, the condition is equivalent to $D_{x_i}D_{x_i}(\psi^\frac{1}{\beta+1}u^\frac{\a}{\beta+1})+\psi^\frac{1}{\beta+1}u^\frac{\a}{\beta+1}>0$ since $\rho=\sqrt{\langle X,X\rangle}$ is independent with $x_i$. Note that
 \begin{align*}
D_{x_i}u=&D_{x_i}\la X,\nu\ra=\la X,D_{x_i}\nu\ra,\\
D_{x_i}D_{x_i}u=&D_{x_i}\la X,D_{x_i}\nu\ra=\la X,D_{x_i}D_{x_i}\nu\ra=\la X,-\nu\ra=-u.
 \end{align*}
This means $$\frac{D_iD_i\psi}{\psi}-\frac{\beta}{\beta+1}(\frac{D_i\psi}{\psi})^2+\frac{2\a}{\beta+1}\frac{D_i\psi D_{x_i}u}{\psi u}+\beta+1-\a+\frac{\a(\a-\beta-1)}{\beta+1}(\frac{D_{x_i}u}{u})^2>0.$$
We have
\begin{align*}
&\frac{D_iD_i\psi}{\psi}+\frac{\a-\beta}{\beta+1-\a}\(\frac{D_i\psi}{\psi}\)^2+\(\beta+1-\a\)\\
&-\frac{\a}{\beta+1}\(\sqrt{\beta+1-\a}\frac{D_{x_i}u}{u}-\frac{1}{\sqrt{\beta+1-\a}}\frac{D_i\psi}{\psi}\)^2>0
\end{align*}
if $\a\le0$. %And if $\a>\beta+1$, we have
%\begin{align*}
%	&\frac{D_iD_i\psi}{\psi}+\frac{\a-\beta}{\beta+1-\a}\(\frac{D_i\psi}{\psi}\)^2+\(\beta+1-\a\)\\
%	&+\frac{\a}{\beta+1}\(\sqrt{\a-\beta-1}\frac{D_{x_i}u}{u}+\frac{1}{\sqrt{\a-\beta-1}}\frac{D_i\psi}{\psi}\)^2>0.
%\end{align*}
Thus we let $\a\le0$ and $(D_iD_j\psi^{\frac{1}{1+\beta-\a}}+\delta_{ij}\psi^{\frac{1}{1+\beta-\a}})>0$ %or $\a>\beta+1$ and $(D_iD_j\psi^{\frac{1}{1+\beta-\a}}+\delta_{ij}\psi^{\frac{1}{1+\beta-\a}})<0$ 
in \cite{DL3}.

 Moreover for $G=\psi(\nu)u^{\a-1}\rho^\delta$, condition (i) in Theorem \ref{t1.2} is equivalent to $\a+\delta+\beta<1$. We have proved similar results in \cite{DL3} if $\psi=1$ or $F=\sigma_{k}^\frac{1}{k}$. For general spherical function $f(\nu)$ and curvature function $F(\lambda_i)$, and arbitrarily strictly convex initial hypersurface, by Theorem \ref{t1.2}, (\ref{1.3})--(\ref{1.8}), we have the following corollary, which covers lots of known results, e.g. \cite{IM,DL,SWM}, and partial results in \cite{DL3,BIS3,CL}. We promote subconvergence to full convergence by Theorem \ref{t5.1}.

%Meanwhile, we can also let $\a\ge\beta+1$ and $(D_iD_j\psi^{\frac{1}{1+\beta-\a}}+\delta_{ij}\psi^{\frac{1}{1+\beta-\a}})<0$ to derive a solution of $\sigma_{k}(D^2u+u\Rmnum{1})=\psi u^{p-1}(u^2+|Du|^2)^\frac{k+1-q}{2}$ with $q\le p\le0$. This case is left out in \cite{DL3}.

\begin{corollary}\label{c1.2}
Let $F\in C^2(\Gamma_+)\cap C^0(\p\Gamma_+)$ satisfy Assumption \ref{a1.1}, and let $M_0$ be a closed, smooth, uniformly convex hypersurface in $\mathbb{R}^{n+1}$, $n\ge2$, enclosing the origin. If $\a+\delta+\beta<1$, $\beta>0$,  and
$D_{x_i}D_{x_i}(\psi^\frac{1}{\beta+1}u^\frac{\a}{\beta+1})+\psi^\frac{1}{\beta+1}u^\frac{\a}{\beta+1}>0$,
then flow (\ref{1.3}) has a unique smooth and uniformly convex solution $M_t$ for all time $t>0$. After rescaling $X\to \varphi^{-1}(t)X$ defined in (\ref{1.5}), the hypersurface $\widetilde M_t=\varphi^{-1}M_t$ converges smoothly to a smooth solution of $\psi(\nu)u^{\a-1}\rho^\delta F^\beta=1$.

\end{corollary}

We see that condition (i) and (ii) are crucial to Theorem \ref{t1.2} if $F\ne\sigma_{n}^\frac{1}{n}$. Condition (iii) in Theorem \ref{t1.2} is used to prove the convergence. Thus we hope that condition (iii) can be removed, i.e. we hope arbitrarily strictly convex initial hypersurface can smoothly converge along flow (\ref{1.2}), which has important implications for other applications of flow, such as proving geometric inequalities.

Let $G=\varphi(\nu,u)$ and $F^k$ is divergence-free (i.e. $\sum_{i}D_i(F^{k})^{ij}=0$) for some positive integer $k$, we can drop condition (iii). Furthermore, if $G=\varphi(\nu,u)\phi(X)$ and $F=\sigma_n^\frac{1}{n}$, we can drop condition (ii) and (iii).

\begin{theorem}\label{t1.3}
Let $F\in C^2(\Gamma_+)\cap C^0(\p\Gamma_+)$ satisfy Assumption \ref{a1.1}, and let $M_0$ be a closed, smooth, uniformly convex hypersurface in $\mathbb{R}^{n+1}$, $n\ge2$, enclosing the origin. If $F^k$ is divergence-free and $G=\varphi(\nu,u)$ satisfies
	
$(i)$ ${\lim\sup}_{s\rightarrow+\infty}[\varphi(\nu,s)s^\beta]<1<\lim\inf_{s\rightarrow0^+}[\varphi(\nu,s)s^\beta]$, $\beta>0$,
	
$(ii)$ $(D_{x_i}D_{x_j}(\varphi u)^\frac{1}{\beta+1}+(\varphi u)^\frac{1}{\beta+1}\delta_{ij})>0$,
	
\noindent then flow (\ref{1.2}) has a unique smooth strictly convex solution $M_t$ for all time $t>0$, and a subsequence of $M_t$ converges in $C^\infty$-topology to a positive, smooth, uniformly convex solution to $\varphi(\nu,u)F^{\beta}=1$.

%In particular, if $F=\sigma_{n}^\frac{1}{n}$, the condition (2) can be dropped.
\end{theorem}

\textbf{Remark:}  The similar results have been derived in \cite{BIS3} and \cite{JLL}. In \cite{BIS3},  Bryan-Ivaki-Scheuer derive this result if $F=\sigma_{n}^{\frac 1n}$ and $\varphi(\nu,u)=\psi(\nu)\varphi(u)$. The range of curvature functions is broadened in this theorem and usually general curvature functions are more difficult to handle than $\sigma_{n}$. Ju-Li-Liu in \cite{JLL} give a different condition and derive a solution of $\psi(\nu)\varphi(u)\sigma_k=c$. In Orlicz type problems the constant $c$ can't be set to $1$ in general.

\begin{theorem}\label{t1.4}
Let $M_0$ be a closed, smooth, uniformly convex hypersurface in $\mathbb{R}^{n+1}$, $n\ge2$, enclosing the origin. Let $G=\varphi(\nu,u)\phi(X)$ and $F=\sigma_n^\frac{1}{n}$. Suppose 	
$$\lim\sup_{s\rightarrow+\infty}[\max_{y=sx}\varphi(x,s)\phi(y)s^\beta]<1<\lim\inf_{s\rightarrow0^+}[\min_{y=sx}\varphi(x,s)G(y)s^\beta].$$
Then flow (\ref{1.2}) has a unique smooth strictly convex solution $M_t$ for all time $t>0$, and a subsequence of $M_t$ converges in $C^\infty$-topology to a positive, smooth, uniformly convex solution to $\varphi(\nu,u)\phi(X)\sigma_{n}^\frac{\beta}{n}=1$.
\end{theorem}

\textbf{Remark:} Comparing with Theorem \ref{t1.3}, we limit the curvature function, and as compensation, we increase the range of $G$. A weaker condition ia given and thus this theorem covers the result in \cite{LL}.

%By the convergence of flow (\ref{1.2}) we can derive the existence results of the dual Orlicz Christoffel-Minkowski type problem mentioned below. As applications, anisotropic flows \cite{B5,CL,IM,IM3,LSW} etc.. usually provided alternative proofs and smooth category approach of the existence of solutions to elliptic PDEs arising in convex body geometry. In these references, the authors usually consider expanding or contracting flows of the convex hypersurfaces at the speeds of $\psi u^\a\sigma_{k}^\beta(\l)$ or $-\psi \rho^\a\sigma_{n}^\beta(\k)$ respectively, where $\psi$ is a smooth positive function on $\mS^n$, $\k$ are the principal curvature of the hypersurface, $\l=\frac{1}{\k}$ are the principal curvature radii of the hypersurface.
 %One advantage of this method is that there is no need to employ the constant rank theorem. For example, if $f^{-1}=\sigma_{k}^\frac{1}{k}(\l)$, by the bound of $f^{-1}=\sigma_{k}^\frac{1}{k}(\l)=(\frac{\sigma_{n-k}}{\sigma_{n}})^\frac{1}{k}(\k)$ one has that the hypersurfaces naturally preserve convexity.

Under flow (\ref{1.2}), if we parameterize the hypersurface $M_t$ by inverse Gauss map $X(x,t): \mS^n\times[0,T^*)\to \mR^{n+1}$, by \cite{DL} Section 2 the support function $u$ satisfies
\begin{equation}\label{1.9}
	\begin{cases}
		&\frac{\partial u}{\partial t}=(G(x,X)F^{\beta}(\l_i)-1)u,\\
		&u(\cdot,0)=u_0.
	\end{cases}
\end{equation}

Therefore the asymptotic behavior of the flow means existence of solutions of dual Orlicz Christoffel-Minkowski type problem mentioned below. In order to prove the above theorems, we shall establish the a priori estimates for the parabolic equation (\ref{1.9}).

Convex geometry plays important role in the development of fully nonlinear partial differential equations. Curvature measures and area measures are two main subjects in convex geometry. The Minkowski problem and Christoffel-Minkowski problem are problems of prescribing a given $n$th and $k$th area measure respectively.
The classical Minkowski problem, the Christoffel-Minkowski problem, the $L^p$ Minkowski problem, the $L^p$ Christoffel-Minkowski problem in general, are beautiful examples of such interactions (e.g., \cite{FWJ2,GM,LE,LO}). Very recently, the $L^p$ dual curvature measures \cite{LYZ2,HLY2} were developed to the Orlicz case \cite{GHW1,GHW2}, which
unified more curvature measures and were named dual Orlicz curvature measures.
These curvature measures are of central importance in the dual Orlicz-Brunn-Minkowski theory, and the corresponding Minkowski problems are called the dual
Orlicz-Minkowski problems. This problem can be reduced to derive the solution of
\begin{equation}\label{1.14}
	c\varphi(u)G(\bar\nn u)\det(D^2u+u\Rmnum{1})=f \qquad \text{ on }\mS^n.
\end{equation}
There have a few results, e.g. \cite{GHW1,GHW2,LL,CLL}.

In this paper, we call
$$W_F(K)=\frac{1}{k+1}\int_{\mS^n}u_{K}F^{k}(\l_{\p K})dx$$
 modified quermassintegrals if $F$ satisfies Assumption \ref{a1.1} and $F^k$ is divergence-free. Here $u_{K}$ and $\l_{\p K}$ are respectively the support function and principal curvature radius of the boundary $\p K$ for any $K\in\cK_o^{n+1}$, the class of compact convex sets in $\mR^{n+1}$ containing the origin. To extend the Orlicz Brunn-Minkowski theory, we shall derive a vatiational formula of $W_F(K)$ with Orlicz linear combination and consider the corresponding prescribed area measure problem (Orlicz Christoffel-Minkowski type problem), which generalizes the Christoffel-Minkowski setting \cite{XZ}. Firstly, we introduce the Orlicz linear combination $+_\varphi(K_1,\cdots,K_m,\a_1,\cdots,\a_m)$ posed in \cite{GHW}.  More generally, for $K_j\in\cK_o^{n+1}$, if $\a_j\ge0$ and $\varphi_j$ are suitable Orlicz functions ( i.e. $\varphi_j$ are convex, strictly
 increasing with $\varphi_j(0)=0$, $\varphi_j(1)=1$), $j=1,\cdots,m$, the Orlicz linear combination $+_\varphi(K_1,\cdots,K_m,\a_1,\cdots,\a_m)$ is defined by
\begin{equation*}
u_{+_\varphi(K_1,\cdots,K_m,\a_1,\cdots,\a_m)}(x)=\inf\{\l>0:\sum_{j=1}^m\a_j\varphi_j\(\frac{u_{K_j}(x)}{\l}\)\le1\},
\end{equation*}
for all $x\in\mR^{n+1}$. To derive the variational formula, we focus on the case $m=2$. For $K,L\in\cK_o^{n+1}$ and $\a,\beta\ge0$, the Orlicz linear combination $+_\varphi(K,L,\a,\beta)$ can be defined via the implicit equation
\begin{equation*}
	\a\varphi_1\(\frac{u_K(x)}{u_{+_\varphi(K,L,\a,\beta)}(x)}\)+\beta\varphi_2\(\frac{u_L(x)}{u_{+_\varphi(K,L,\a,\beta)}(x)}\)=1,
\end{equation*}
if $\a u_K(x)+\beta u_L(x)>0$, and by $u_{+_\varphi(K,L,\a,\beta)}(x)=0$ if $\a u_K(x)+\beta u_L(x)=0$ for all $x\in\mR^{n+1}$.
When $\varphi_1(x)=\varphi_2(x)=x^p$, it reverts back to $L^p$ linear combination. For convenience, we shall write $K+_{\varphi,\eps}L$ instead of $+_\varphi(K,L,1,\eps)$ for $\eps\ge0$. Now we can introduce the Orlicz-mixed modified quermassintegrals which are defined by the variation of modified quermassintegrals.
\begin{definition}\label{d1.6}
Let $F$ satisfy Assumption \ref{a1.1} and $F^k$ be divergence-free. We
define the Orlicz-mixed modified quermassintegral of two smooth strictly convex bounded domains $K\in\cK_{oo}^{n+1}$, $L\in\cK_{o}^{n+1}$ by
\begin{equation*}
W_{\varphi,F}(K,L)=\lim_{\eps\rightarrow0^+}\dfrac{W_F(K+_{\varphi,\eps}L)-W_F(K)}{\eps},
\end{equation*}
where $\cK_{oo}^{n+1}$ denotes the class of compact convex sets in $\mR^{n+1}$ containing the origin in their interior.
\end{definition}
Then we have the following integral representation of $W_{\varphi,F}(K,L)$.
\begin{theorem}\label{t1.7}
Let $F,K,L$ satisfy the assumptions in Definition \ref{d1.6}. Then the
integral representation of the Orlicz-mixed modified quermassintegral $W_{\varphi,F}(K,L)$ is given by
\begin{equation*}
W_{\varphi,F}(K,L)=\int_{\mS^n}\dfrac{u_K(x)}{(\varphi_1)'_l(1)}\varphi_2\(\frac{u_L(x)}{u_K(x)}\)F^k(\l_{\p K})dx,
\end{equation*}
where $(\varphi_1)'_l$ is the left derivative of $\varphi_1$.
\end{theorem}

We obtain by flow method the following inequalities involving the modified quermassintegrals, which can lead to $L^p$ Minkowski type inequalities (see Section 6 for details).
\begin{corollary}\label{t1.8} (1)
Let $F,\varphi$ satisfy the assumptions in Theorem \ref{t1.3}. Then for any $K\in\cK_{oo}^{n+1}$, there exists $L\subset\mR^{n+1}$ satisfied that $\p L$ is a smooth, strictly convex solution of $\varphi(\nu,u)F^{\beta}=1$, such that
\begin{equation*}
W_F(K)-W_F(L)\le\int_{\mS^n}\int_{u_{L}}^{u_{K}}\varphi^{-\frac{k}{\beta}}(x,s)dsdx
\end{equation*}
with equality if and only if $K=L$.

(2) Let $\phi,\varphi$ satisfy the assumptions in Theorem \ref{t1.4}. Then for any $K\in\cK_{oo}^{n+1}$, there exists $L\subset\mR^{n+1}$ satisfied that $\p L$ is a smooth, strictly convex solution of $\varphi(\nu,u)\phi(X)\sigma_{n}^\frac{\beta}{n}=1$, such that
\begin{equation*}
\int_{\mS^n}\int_\eps^{\rho_K}\phi^\frac{n}{\beta}(\xi,s)s^ndsd\xi-\int_{\mS^n}\int^{\rho_L}_\eps\phi^\frac{n}{\beta}(\xi,s)s^ndsd\xi\le \int_{\mS^n}\int_{u_L}^{u_K}\varphi^{-\frac{n}{\beta}}(x,s)dsdx
\end{equation*}
for $\forall \eps>0$,
with equality if and only if $K=L$.
\end{corollary}

By Theorem \ref{t1.7}, the modified Orlicz surface area measure of a smooth uniformly convex bounded domain $K\subset\mR^{n+1}$ can be defined by
\begin{equation*}
dS_{\varphi,F}(K,x)=u_K(x)\varphi\(\frac{1}{u_K(x)}\)F^k(\l_{\p K})dx.
\end{equation*}
In particular, if $F^k=\sigma_{k}$, it reverts back to the Orlicz $k$-th surface area measure. This motivates us to raise the following Orlicz Christoffel-Minkowski type problem, which is a nonlinear elliptic PDE on $\mS^n$.
\begin{problem}
(\textbf{Orlicz Christoffel-Minkowski type problem}). Given a smooth positive function $f(x)$ defined on $\mS^n$, what are necessary and sufficient conditions for $f(x)$, such that there exists a smooth strictly convex bounded domain $K\subset\mR^{n+1}$ satisfying
$$dS_{\varphi,F}(K,x)=f(x)dx.$$
That is, finding a smooth positive solution $u$ to
$$u(x)\varphi\(\frac{1}{u(x)}\)F^k([D^2u(x)+u(x)\Rmnum{1}])=f(x),$$
such that $[D^2u(x)+u(x)\Rmnum{1}]>0$ for all $x\in\mS^n$.
\end{problem}

Inspired by \cite{GRW,GHW1,GHW2}, we shall consider the dual Orlicz Christoffel-Minkowski type problem. Since $X=Du+ux,$ as before, this problem can be reduced to the following nonlinear PDE:
\begin{equation}\label{1.16}
	G(x,X)F(D^2u+u\Rmnum{1})=c\quad\text{ on } \mS^n,
\end{equation}
or
\begin{equation}\label{1.15}
	G(x,u,Du)F(D^2u+u\Rmnum{1})=c\quad\text{ on } \mS^n.
\end{equation}

%(\ref{1.15}) or (\ref{1.16}) can be converted into $L^p$ Christoffel-Minkowski problem, $L^p$ dual Minkowski problem, $L^p$ dual Christoffel-Minkowski problem,
%dual Orlicz Minkowski problem or Orlicz Christoffel-Minkowski  problem if $G=\psi(x)u^\frac{1-p}{k}$, $F=\sigma_{k}^\frac{1}{k}$ or $G=\psi(x)u^\frac{1-p}{n}\rho^\frac{q-n-1}{2n}$, $F=\sigma_{n}^\frac{1}{n}$ or  $G=\psi(x)u^\frac{1-p}{k}\rho^\frac{q-k-1}{2k}$, $F=\sigma_{k}^\frac{1}{k}$ or  $G=\psi(x)\varphi(u)\phi(X)$, $F=\sigma_{n}^\frac{1}{n}$ or $G=\psi(x)\varphi(u)$, $F=\sigma_{k}^\frac{1}{k}$ respectively.
Note that the case $F\ne\sigma_{n}^\frac{1}{n}$ is generally harder to deal with than $F=\sigma_{n}^\frac{1}{n}$. If $G$ has a particular form, some previous known results of this problem can be seen in \cite{GM,HMS,GX,LYZ2,CL,HZ,CCL,CHZ,BF,DL3} etc..

By Theorem \ref{t1.2}, we prove the following existence results of the dual Orlicz Minkowski problem and dual Orlicz Christoffel-Minkowski type problem by the asymptotic behavior of the anisotropic curvature flow.
\begin{corollary}\label{t1.5}
Let $F$ satisfies Assumption \ref{a1.1} and
 $$\lim\sup_{s\rightarrow+\infty}[\max_{y=sx}G(x,y)s]<1<\lim\inf_{s\rightarrow0^+}[\min_{y=sx}G(x,y)s].$$
 Suppose either

 (\rmnum{1}) $\exists \beta>0$ such that  $(D_{x_i}D_{x_j}(G^\beta u)^\frac{1}{\beta+1}+(G^\beta u)^\frac{1}{\beta+1}\delta_{ij})>0$;

or (\rmnum{2}) $F=\sigma_{n}^\frac{1}{n}$.

\noindent Then there exists a  positive, smooth, uniformly convex solution to  (\ref{1.16}) with $c=1$.
\end{corollary}
\textbf{Remark}: (1) If $F=\sigma_{n}^\frac{1}{n}$, our assumption is weaker than the condition in \cite{LL}. Note that  the power of curvature function is different between this theorem and that in \cite{LL}. Compared with \cite{LL}, we extend $\sigma_{n}$ and $f(x)\varphi(u)G(X)$ to general $F$ and $G(x,X)$ respectively.

(2) If $G=\psi u^{\a-1}\rho^\delta$ and $F=\sigma_{k}^\frac{1}{k}$, the problem can be converted into the $L^p$ dual Christoffel-Minkowski problem considered in \cite{DL3}, i.e.
\begin{equation}\label{1.12}
\sigma_{k}(D^2u+u\Rmnum{1})=\psi(x)u^{p-1}(u^2+\vert Du\vert^2)^\frac{k+1-q}{2}.
\end{equation}
We derive an existence result of this problem with $p>q$, $p>1$. If $F=\sigma_{n}^\frac{1}{n}$, We derive an existence result of $L^p$ dual Minkowski problem with $p>q$. If we let $p$ or $q$ be a special value, our argument covers lots of well-known results, such as \cite{IM,HMS,LL,HZ,SWM}, partial results in \cite{DL3,BIS3,CL,LSW} etc..
%For example, in \cite{IM,HMS,SWM} they proved the existence of the solutions of (\ref{1.12}) with $q=k+1$, $p\ge k+1$; in \cite{HZ} they proved the existence of the solutions of (\ref{1.12}) with  $p>q$ and $k=n$; in \cite{LL} they proved the existence of the solutions of (\ref{1.16}) with $G=\psi(x)\varphi(u)\phi(X)$ and $F=\sigma_{n}^\frac{1}{n}$ covered by (2) of Theorem \ref{t1.5}; in \cite{CL} they proved the existence of the solutions of (\ref{1.12}) with  $p\ge q$ and $k=n$; in \cite{DL3} we proved the existence of the solutions of (\ref{1.12}) with $p>q$, $p>1$ or $q<p<0$; in \cite{LSW} they proved the existence of the solutions of (\ref{1.12}) with  $p=0\ge q$ and $k=n$; in \cite{BIS3} they proved the existence of the solutions of (\ref{1.16}) with $G=\psi(x)\varphi(u)$ and $F=\sigma_{n}^\frac{1}{n}$ covered by (2) of Theorem \ref{t1.5}; etc.. %Since we derive the results of non-homogeneous case, the case $p=q$ couldn't be solved is easy to understand.

%\cite{IM,HMS,LL,HZ,SWM}, partial results in \cite{DL3,BIS3,CL,LSW}.

On the other hand, Guan-Ren-Wang in \cite{GRW} raised a question whether the $C^2$ estimate of general prescribed curvature measure problem
\begin{equation}\label{19}
\sigma_{k}(\k)=G(\nu,X)
\end{equation}
can be derived. This problem with the general form $G(\nu,X)$ is a longstanding problem, has great significance and is worth considering. An example is given in \cite{GRW} that the $C^2$ estimate for $F=\frac{\sigma_{k}}{\sigma_{l}}$ fails, where $1\le l< k\le n$. Compared with Theorem 1.6 in \cite{GRW}, the existence of the solution to this equation can be obtained with additional condition (i) and curvature function satisfying Assumption \ref{a1.1}. Since the hypersurfaces $M_t$ are strictly convex, condition (\rmnum{1}) is necessary by Remark (2) of Theorem \ref{t1.2}. Thus we give a partial answer to the question raised in \cite{GRW}, i.e., we have the following corollary about prescribed curvature measure problem.

\begin{corollary}\label{c1.6}
Suppose $F(\l)$ and $G$ satisfy the assumptions of Corollary \ref{t1.5}, then equation
\begin{equation*}
F_*(\k)=G(\nu,X)
\end{equation*}
has a strictly convex, smooth solution.
\end{corollary}

The rest of the paper is organized as follows. We first recall some notations and known results in Section 2 for later use. In Section 3,
we establish the $C^2$ estimates for a large ranges of parabolic equations. In Section 4, by establishing the $C^0$ estimate and the a priori estimate, we then show the convergence of these flows and complete the proofs of Theorems \ref{1.3}, \ref{1.5} and \ref{1.6}. In Section 5, we provide two uniqueness results in some special cases, and complete the proof of Corollary \ref{1.4}. Finally, in Section 6, we first derive the variational formula of the modified quermassintegrals to prove Theorem \ref{1.8}. As applications of Theorems \ref{1.5} and \ref{1.6}, inequalities in Corollary \ref{t1.8} will be proved, and more special cases are discussed.

\section{Preliminary}
\subsection{Intrinsic curvature}
We now state some general facts about hypersurfaces. Quantities for hypersurface $M$ will be denoted by $(g_{ij})$, $(R_{ijkl})$ etc., where Latin indices range from $1$ to $n$.

Let $\nabla$, $\bar\nabla$ and $D$ be the Levi-Civita connection of $g$, the standard metric of $\mR^{n+1}$ and the Riemannian metric $e$ of $\mathbb S^n$  respectively. All indices appearing after the semicolon indicate covariant derivatives. The $(1,3)$-type Riemannian curvature tensor is defined by
\begin{equation*}\label{2.1}
	R(U,Y)Z=\nabla_U\nabla_YZ-\nabla_Y\nabla_UZ-\nabla_{[U,Y]}Z,
\end{equation*}
and in a local frame $\{e_i\}$,
\begin{equation*}\label{2.2}
	R(e_i,e_j)e_k={R_{ijk}}^{l}e_l,
\end{equation*}
where we use the summation convention (and will henceforth do so). The so-called Ricci identity, reads
\begin{equation*}\label{2.3}
	Y_{;ij}^k-Y_{;ji}^k=-{R_{ijm}}^kY^m
\end{equation*}
for all vector fields $Y=(Y^k)$. We also denote the $(0,4)$ version of the curvature tensor by $R$,
\begin{equation*}\label{2.4}
	R(W,U,Y,Z)=g(R(W,U)Y,Z).
\end{equation*}
\subsection{Extrinsic curvature}
The induced geometry of $M$ is governed by the following relations. The second fundamental form $h=(h_{ij})$ is given by the Gaussian formula
\begin{equation*}\label{2.5}
	\bar\nabla_ZY=\nabla_ZY-h(Z,Y)\nu,
\end{equation*}
where $\nu$ is a local outer unit normal field. Note that here (and in the rest of the paper) we will abuse notation by disregarding the necessity to distinguish between a vector $Y\in T_pM$ and its push-forward $X_*Y\in T_p\mathbb{R}^{n+1}$. The Weingarten endomorphism $A=(h_j^i)$ is given by $h_j^i=g^{ki}h_{kj}$, and the Weingarten equation
\begin{equation*}\label{2.6}
	\bar\nabla_Y\nu=A(Y),
\end{equation*}
holds there, or in coordinates
\begin{equation*}\label{2.7}
	\nu_{;i}^\alpha=h_i^kX_{;k}^\alpha.
\end{equation*}
We also have the Codazzi equation in $\mathbb{R}^{n+1}$
\begin{equation*}\label{2.8}
	\nabla_Wh(Y,Z)=\nabla_Zh(Y,W)
\end{equation*}
or
\begin{equation*}\label{2.9}
	h_{ij;k}=h_{ik;j},
\end{equation*}
and the Gauss equation
\begin{equation*}\label{2.10}
	R(W,U,Y,Z)=h(W,Z)h(U,Y)-h(W,Y)h(U,Z)
\end{equation*}
or
\begin{equation*}\label{2.11}
	R_{ijkl}=h_{il}h_{jk}-h_{ik}h_{jl}.
\end{equation*}

\subsection{Convex hypersurface parametrized by the inverse Gauss map}
Let $M$ be a smooth, closed and uniformly convex hypersurface in $\mR^{n+1}$. Assume that $M$ is parametrized by the inverse Gauss map $X: \mS^n\to M\subset \mR^{n+1}$ and encloses origin. The support function $u: \mS^n\to \mR^1$ of $M$ is defined by $$u(x)=\sup_{y\in M}\la x,y\ra.$$ The supremum is attained at a point $y=X(x)$ because of convexity, $x$ is the outer normal of $M$ at $y$. Hence $u(x)=\la x,X(x)\ra$. Then we have
\begin{equation*}
X=Du+ux\quad \text{    and    } \quad\rho=\sqrt{u^2+|Du|^2}.
\end{equation*}

%Let $e_1,\cdot\cdot\cdot,e_n$ be a smooth local orthonormal frame field on $\mS^n$, and $D$ be the covariant derivative with respect to the standard metric $e_{ij}$ on $\mS^n$. Denote by $g_{ij}, g^{ij}, h_{ij}$ the induced metric, the inverse of the induced metric, and the second fundamental form of $M$, respectively. Then
The second fundamental form of $M$ in terms of $u$ is given by
$$h_{ij}=D_iD_ju+ue_{ij},$$
% To compute the metric $g_{ij}$ of $M$ we use the Gauss-Weingarten relations $D_ix=h_{ik}g^{kl}D_lX$, from which we obtain$$e_{ij}=<D_ix,D_jx>=h_{ik}g^{kl}h_{jm}g^{ms}<D_lx,D_sx>=h_{ik}h_{jl}g^{kl}.$$
%Since $M$ is uniformly convex, $h_{ij}$ is invertible and the inverse is denoted by $h^{ij}$, hence $g_{ij}=h_{ik}h_{jk}$.
and the principal radii of curvature are the eigenvalues of the matrix
\begin{equation}
	b_{ij}=h^{ik}g_{jk}=h_{ij}=D_{ij}^2u+u\delta_{ij}.
\end{equation}
The proof can be seen in Urbas \cite{UJ}. We see that flow (\ref{1.2}) is equivalent to the following initial value problem of the support function $u$ by \cite{DL} Section 2,
\begin{equation}\label{2.19}
	\begin{cases}
		&\frac{\p u}{\p t}=\(GF^\beta([D^2u+u\uppercase\expandafter{\romannumeral1}])-1\)u\; \text{ on } \mS^n\times[0, T^*),\\
		&u(\cdot,0)=u_0,
	\end{cases}
\end{equation}
where $\uppercase\expandafter{\romannumeral1}$ is the identity matrix, and $u_0$ is the support function of $M_0$. For any $F$ satisfying
Assumption \ref{a1.1}, there holds
\begin{equation}
	[F^{ij}]>0 \text{ on } \Gamma^+,
\end{equation}
which yields that the equation (\ref{2.19}) is parabolic for admissible solutions. We also have the formula
\begin{equation}\label{2.22}
\frac{\p_t\rho}{\rho}(\xi)=\frac{\p_tu}{u}(x),
\end{equation}
which can be also found in \cite{LSW,CL}.

It is well-known that the determinants of the Jacobian of radial Gauss mapping $\sA$ and reverse radial Gauss mapping $\sA^*$ of a convex body $\Om$ in $\mR^{n+1}$ are given by, see e.g. \cite{HLY2,CCL,LSW},
\begin{equation}\label{2.27}
	\vert Jac\sA\vert(\xi)=\vert\frac{dx}{d\xi}\vert=\frac{\rho^{n+1}(\xi)K(\vec{\rho}(\xi))}{u(\sA(\xi))},
\end{equation}
where $\vec{\rho}(\xi)=\rho(\xi)\xi$, and
\begin{equation}\label{2.24}
	\vert Jac\sA^*\vert(x)=\vert\frac{d\xi}{dx}\vert=\dfrac{u(x)}{\rho^{n+1}(\sA^*(x))K(\nu^{-1}_\Om(x))}.
\end{equation}

Before closing this section, we list the following basic properties for any given $\Om$ containing the origin.
\begin{lemma}\cite{CL}\label{l2.2}
Suppose $\Om$ contains the origin. Let $u$ and $\rho$ be the support function and radial function of $\Om$, and $x_{\max}$ and $\xi_{\min}$ be two points such that $u(x_{\max})=\max_{\mS^n}u$ and $\rho(\xi_{\min})=\min_{\mS^n}\rho$. Then
\begin{align}
\max_{\mS^n}u=\max_{\mS^n}\rho\quad \text{    and    }\quad \min_{\mS^n}u=\min_{\mS^n}\rho,\label{2.29}\\
u(x)\ge x\cdot x_{\max}u(x_{\max}),\quad \forall x\in\mS^n,\label{2.30}\\
\rho(\xi)\xi\cdot\xi_{\min}\le\rho(\xi_{\min}),\quad \forall \xi\in\mS^n.\label{2.31}
\end{align}
\end{lemma}

\section{Curvature estimates for general cases}
In this section, we shall establish the $C^2$ estimates for large ranges of parabolic equations. We hope that these estimates can be used more widely. We only need to derive the $C^0$ estimates to get the long-time existence in future work. We give an assumption about curvature function.
\begin{assumption}\label{a3.1}
	Let $\Gamma_+=\{(\l_i)\in\mathbb{R}^n:\l_i>0\}$ be a symmetric, convex, open cone.
	
(\rmnum{1}) Suppose that $F$ is positive in $\Gamma_+$, homogeneous of degree $1$ with $F^\beta(1,\cdots,1)=1,$
	$\frac{\p F}{\p\l_i}>0;$

(\rmnum{2}) Suppose $F$ is inverse concave.
\end{assumption}

Firstly, we shall derive the $C^1$ estimates if we have $C^0$ estimates. Throughout the paper, we say that $u\in C^2(\mS^n)$ is uniformly convex if the matrix $[D^2u+u\Rmnum{1}]$ is positive-definite. Given a smooth positive function $u$ on $\mS^n$ such that $[D^2u+u\Rmnum{1}]$ is positive-definite, there is a unique smooth strictly convex hypersurface $M$ given by
$$M=\{X(x)\in\mR^{n+1}|X(x)=u(x)x+Du(x),x\in\mS^n\},$$
whose support function is $u$. For a family of parabolic equations of $u(\cdot,t)$, we can consider hypersurfaces $M_t$ whose support functions are $u(\cdot,t)$. Thus the uniformly bound of $|Du(x,t)|$ is derived directly by Lemma \ref{l2.2} and $\rho=\sqrt{u^2+\vert Du\vert^2}$  if we have the uniformly positive bounds of $u$ from above and below.

Next we have the bounds of curvature function $F$.
\begin{lemma}\label{l3.3}
	Let $u(\cdot,t)$ be a positive, smooth and uniformly convex solution to
	\begin{equation*}
		\frac{\p u}{\p t}=\phi(t)uG(x,u,Du)F^\beta(D^2u+u\Rmnum{1})-\eta(t)u\quad\text{ on }\mS^n\times[0,T),
	\end{equation*}
with $\beta>0$, where $G(x,u,Du): \mS^n\times(0,\infty)\times\mR^n\rightarrow(0,\infty)$ is a smooth function, and $F$ satisfies (\rmnum{1}) of Assumption \ref{a3.1}. If
	\begin{align*}
		\frac{1}{C_0}\le u(x,t)\le C_0,\qquad\qquad &\forall(x,t)\in\mS^n\times[0,T),\\
		\eta(t)\ge	\frac{1}{C_0},\qquad\qquad&\forall t\in[0,T),\\
		0\le\phi(t)\le C_0,\qquad \qquad&\forall t\in[0,T),
	\end{align*}
	for some constant $C_0>0$, then
	\begin{equation*}
		F(D^2u+u\Rmnum{1})(\cdot,t)\ge C^{-1},\qquad \forall t\in[0,T),
	\end{equation*}
	where $C$ is a positive constant depending only on $\beta,C_0,u(\cdot,0)$, $\vert G\vert_{L^\infty(U)}$, $\vert1/G\vert_{L^\infty(U)}$, $\vert G\vert_{C^1_{x,u,Du}(U)}$ and $\vert G\vert_{C^2_{x,u,Du}(U)}$ where $U=\mS^n\times[1/C_0,C_0]\times B^n_{C_0}$ ($B^n_{C_0}$ is the ball
	centered at the origin with radius $C_0$ in $\mR^n$).
\end{lemma}
\begin{proof}
We consider hypersurfaces $M_t$ whose support function are $u(\cdot,t)$ and use $G=G(x,X)$ in this proof since $G(x,X)$ can be regarded as $G(x,u,Du)$.	We will apply the maximum principle to the auxiliary function
	$$\theta=\log Q-A\frac{\rho^2}{2},$$
	where $Q=uGF^\beta$, $A>0$ is a constant to be determined later. Let  $\cL=\p_t-\frac{\beta \phi(t)QF^{ij}}{F}D^2_{ij}$.
 Then we get
	\begin{equation}\label{3.1}
		\begin{split}
\cL\log Q=&\frac{\p_t Q}{Q}-\frac{\beta\phi(t) F^{ij}}{F}\(Q_{ij}-\frac{Q_iQ_j}{Q}\)\\
=&\frac{\phi Q-\eta u}{u}+\frac{d_XG(\p_tX)}{G}+\frac{\beta F^{ij}}{F}(-\eta h_{ij}+\phi Q\delta_{ij})+\frac{\beta\phi(t) F^{ij}}{F}\frac{Q_iQ_j}{Q}\\
\ge&\frac{\phi Q}{u}-(1+\beta)\eta+\frac{d_XG(\p_tX)}{G},
		\end{split}
	\end{equation}
	where we have used that $[F^{ij}]$ is positive-definite in the last inequality.
\begin{equation}\label{3.2}
	\begin{split}
		\cL(-A\frac{\rho^2}{2})=&-A\cL(\frac{u^2+\vert Du\vert^2}{2})\\
		=&-Au\p_tu-AD^muD_m\p_tu\\
		&+\frac{A\beta \phi(t)QF^{ij}}{F}(uu_{ij}+u_iu_j+u_{mij}u_m+u_{mi}u_{mj})\\
		=&-Au\phi Q+A\eta\rho^2-A\phi F^\beta D^mu(uG)_m-A\beta\phi\frac{Q}{F}D^muF_m\\
		&+\frac{A\beta \phi(t)QF^{ij}}{F}\(uh_{ij}-u^2\delta_{ij}+u_iu_j+(D_mh_{ij}-u_j\delta_{mi})u_m\\
		&+(h_{mi}-u\delta_{mi})(h_{mj}-u\delta_{mj})\)\\
\ge&-A(\beta+1)u\phi Q+A\eta\rho^2-A\phi\frac{Q}{u}\vert Du\vert^2-A\phi\frac{Q}{G}D^mud_XG(\bar\nn_mX)\\
&-A\phi\frac{Q}{G}D^mud_xG(e_m),
	\end{split}
\end{equation}
where we have used $\rho^2=u^2+\vert Du\vert^2$ and again $[F^{ij}]$ is positive-definite in the last inequality. Note that
\begin{align}\label{3x}
D^mu\bar\nn_mX=D^mu\bar\nn_m(Du+ux)=D^muh_{mk}e_k=\rho D\rho(x)
\end{align}
by $D_i\rho(x)=\frac{uu_i+u_ku_{ki}}{\rho}=\frac{u_kh_{ki}}{\rho}$, %By $C^1$ estimates we have $\vert D\rho\vert\le C$, then
%\begin{equation}\label{3.3}
%\vert D^mu\la G_X,\bar\nn_mX\ra\vert=\vert\la G_X,D_k\rho e_k\ra\vert\le C_1
%\end{equation}
%since $G$ is a smooth function and we have the $C^0$ and $C^1$ estimates.
 and that
\begin{align}\label{3.4}
\p_tX=\p_t(ux+Du)=(\phi Q-\eta u)x+D(\phi Q-\eta u)=(\phi Q-\eta u)x+\phi QA\rho D\rho-\eta Du
\end{align}
by $D\theta=\frac{DQ}{Q}-A\rho D\rho=0$.

Combining (\ref{3.1}), (\ref{3.2}), (\ref{3x}) and (\ref{3.4}),  we obtain
	\begin{equation}\label{3.5}
		\begin{split}
\cL\theta\ge&\frac{\phi Q}{u}-(1+\beta)\eta+\frac{d_XG\((\phi Q-\eta u)x+\phi QA\rho D\rho-\eta Du\)}{G}-A(\beta+1)u\phi Q\\
&+A\eta\rho^2-A\phi\frac{Q}{u}\vert Du\vert^2-A\phi\frac{Q}{G}d_XG\(\rho D\rho(x)\)-A\phi\frac{Q}{G}D^mud_xG(e_m)\\
\ge&\frac{\phi Q}{u}-(1+\beta)\eta-C_2(\phi Q+\eta)
-A(\beta+1)u\phi Q+A\eta\rho^2-A\phi\frac{Q}{u}\vert Du\vert^2-AC_3\phi\frac{Q}{G}\\
			=&\(\frac{A}{2}\eta\rho^2+\frac{\phi Q}{u}-C_2\phi Q-(1+\beta+C_2)\eta\)\\
			&+A\(\frac{\eta\rho^2}{2}-(1+\beta)u\phi Q-\phi\frac{Q}{u}\vert Du\vert^2-C_3\phi\frac{Q}{G}\),
		\end{split}
	\end{equation}
where we have used $|d_XG(x)|,|d_XG(Du)|,|d_XG(Du)|\le C_1$ by $C^0$ and $C^1$ estimates.	Note that $F$ and $\theta$ have the same monotonicity. Choose $A>\frac{2}{\min\rho^2}(1+C_2+\beta)$. Thus if $\theta$ tends to $-\infty$ then $Q$ tends to $0$ and the right-hand side becomes positive. This completes the proof by the maximum principle.
\end{proof}
\begin{lemma}\label{l3.4}
	Assume $\beta>0$. Let $u(\cdot,t)$ be a positive, smooth and uniformly convex solution to
	\begin{equation*}
		\frac{\p u}{\p t}=\phi(t)uG(x,u,Du)F^\beta(D^2u+u\Rmnum{1})-\eta(t)u\quad\text{ on }\mS^n\times[0,T),
	\end{equation*}
	where $G(x,u,Du):\mS^n\times(0,\infty)\times\mR^n\rightarrow(0,\infty)$ is a smooth function, $F$ satisfies Assumption \ref{a3.1}. If
	\begin{align*}
		\frac{1}{C_0}\le u(x,t)\le C_0,\qquad\qquad &\forall(x,t)\in\mS^n\times[0,T),\\
		\eta(t)\le	C_0,\qquad\qquad&\forall t\in[0,T),\\
		\phi(t)\max_{\mS^n}F^\beta(D^2u+u\Rmnum{1})\ge \frac{1}{C_0},\qquad \qquad&\forall t\in[0,T),
	\end{align*}
	for some constant $C_0>0$, then
	\begin{equation*}
		F(D^2u+u\Rmnum{1})(\cdot,t)\le C,\qquad \forall t\in[0,T),
	\end{equation*}
	where $C$ is a positive constant depending only on $\beta,C_0,u(\cdot,0)$, $\vert G\vert_{L^\infty(U)}$, $\vert1/G\vert_{L^\infty(U)}$, $\vert G\vert_{C^1_{x,u,Du}(U)}$ and $\vert G\vert_{C^2_{x,u,Du}(U)}$ where $U=\mS^n\times[1/C_0,C_0]\times B^n_{C_0}$ ($B^n_{C_0}$ is the ball
	centered at the origin with radius $C_0$ in $\mR^n$).
\end{lemma}
\begin{proof}
We consider hypersurfaces $M_t$ whose support function are $u(\cdot,t)$ and use $G=G(x,X)$ in this proof since $G(x,X)$ can be regarded as $G(x,u,Du)$.	We will apply the maximum principle to the auxiliary function
	$$\theta=\frac{GF^\beta}{1-\eps\frac{\rho^2}{2}},$$
	where $\eps\le\rho^2\le\frac{1}{\eps}$. Let  $Q=GF^\beta$, $\cL=\p_t-\frac{\beta \phi(t)uQF^{ij}}{F}D^2_{ij}$. Then
	$$(GF^\beta u)_{ij}=Q_{ij}u+Q_iu_j+Q_ju_i+Qu_{ij}.$$
Direct computation yields
	\begin{equation}\label{3.6}
		\begin{split}
			\p_tQ=&\frac{Q}{G}d_XG(\p_tX)+\frac{\beta Q}{F}F^{ij}\((\p_tu)_{ij}+\p_tu\delta_{ij}\)\\
=&\frac{Q}{G}d_XG(\p_tX)+\frac{\beta Q}{F}F^{ij}(\phi Q_{ij}u+\phi Q_iu_j+\phi Q_ju_i+(\phi Q-\eta)h_{ij})\\
=&\beta(\phi Q-\eta)Q+\frac{Q}{G}d_XG(\p_tX)+\frac{\beta \phi uQF^{ij}}{F}Q_{ij}+2\frac{\beta\phi Q}{F}F^{ij}Q_iu_j.
		\end{split}
	\end{equation}
Note that the definition of $Q$ in Lemma \ref{l3.3} is different from that in this lemma. By (\ref{3.2}) we have
	\begin{equation}\label{3.7}
		\begin{split}
\cL(1-\eps\frac{\rho^2}{2})=&-\eps u^2\phi Q+\eps\eta\rho^2-\eps\phi F^\beta D^mu(uG)_m-\eps\beta\phi\frac{uQ}{F}D^muF_m\\
&+\frac{\eps\beta \phi(t)uQF^{ij}}{F}\(uh_{ij}-u^2\delta_{ij}+u_iu_j+(D_mh_{ij}-u_j\delta_{mi})u_m\\
&+(h_{mi}-u\delta_{mi})(h_{mj}-u\delta_{mj})\)\\
=&-\eps(1+\beta)u^2\phi Q+\eps\eta\rho^2-\eps\phi Q\vert Du\vert^2-\eps\phi uF^\beta D^muD_mG\\
&+\frac{\eps\beta \phi(t)uQF^{ij}(h^2)_{ij}}{F}.
		\end{split}
	\end{equation}
	Since $F$ is an inverse concave function, we have $F^{ij}h_{mi}h_{mj}\ge F^2$ by \cite{B4}. At the maximum point, we have
	$$D\theta=\frac{DQ}{1-\eps\frac{\rho^2}{2}}+\frac{\eps\rho QD\rho}{(1-\eps\frac{\rho^2}{2})^2}=0.$$
	Combining (\ref{3x}), (\ref{3.4}), (\ref{3.6}) and (\ref{3.7}), similarly to the proof of Lemma \ref{l3.3} we have
	\begin{equation}\label{3.8}
		\begin{split}
			\frac{1-\eps\frac{\rho^2}{2}}{Q}\cL\theta=&\frac{\cL Q}{Q}-\frac{\cL(1-\eps\frac{\rho^2}{2})}{1-\eps\frac{\rho^2}{2}}\\
			=&\beta(\phi Q-\eta)+\frac{d_XG(\p_tX)}{G}-\frac{2\beta\eps\phi\rho Q}{F(1-\eps\frac{\rho^2}{2})}F^{ij}u_i\rho_j\\
			&-\dfrac{-\eps(1+\beta)u^2\phi Q+\eps\eta\rho^2-\eps\phi Q\vert Du\vert^2-\eps\phi uF^\beta D^muD_mG}{1-\eps\frac{\rho^2}{2}}\\
&-\dfrac{\eps\beta \phi(t)uQF^{ij}(h^2)_{ij}}{F(1-\eps\frac{\rho^2}{2})}\\
			\le&-C_{4}\phi F^{\beta+1}+C_{5}\phi F^{\beta}+C_{6}\\
\le&-C_7\phi F^{\beta+1}+C_{6}\\
\le&-C_8F+C_{6},
		\end{split}
	\end{equation}
where we have used $D_i\rho=\frac{uu_i+u_ku_{ki}}{\rho}=\frac{u_kh_{ki}}{\rho}$ again, $M_t$ is strictly convex, $[F^{ij}]$ is positive-definite, and the bounds of $u,\rho$ and $\vert Du\vert$. At the maximum point we can suppose that $\theta$ tends to infinity, i.e. $F$ tends to $\max_{\mS^n}F$. Thus we use $\phi F^\beta\ge C_0$ in the last inequality. By the maximum principle we complete the proof.
\end{proof}
\textbf{Remark:} The condition $\phi(t)\max_{\mS^n}F^\beta(D^2u+u\Rmnum{1})\ge C_0$ is weaker than $\phi(t)\ge C_0$.

Finally, we can derive the $C^2$ estimates. We use $G=G(t,x,X)$ in the following lemma to make conditions better represented since $G(t,x,X)$ can be regarded as $G(t,x,u,Du)$, where $X$ is denoted by $ux+Du$. Under this representation we consider that $X$ is independent of $x$.
\begin{lemma}\label{l3.5}
	Assume $\beta>0$. Let $u(\cdot,t)$ be a positive, smooth and uniformly convex solution to
	\begin{equation*}
		\frac{\p u}{\p t}=uG(t,x,X)F^\beta(D^2u+u\Rmnum{1})-\eta(t)u\quad\text{ on }\mS^n\times[0,T),
	\end{equation*}
	where $G(t,x,X): [0, \infty)\times\mS^n\times\mR^{n+1}\rightarrow(0,\infty)$ is a smooth function if $\la x,X\ra>0$, $F$ satisfies Assumption \ref{a3.1}, and $(D_{x_i}D_{x_j}(Gu)^\frac{1}{\beta+1}+(Gu)^\frac{1}{\beta+1}\delta_{ij})>0$. If
	\begin{align*}
		\frac{1}{C_0}\le u(x,t)\le C_0,\qquad\qquad &\forall(x,t)\in\mS^n\times[0,T),\\
		\eta(t)\le	C_0,\qquad\qquad&\forall t\in[0,T),\\
		\frac{1}{C_0}\le F(x,t)\le C_0,\qquad\qquad &\forall(x,t)\in\mS^n\times[0,T),
	\end{align*}
	for some constant $C_0>0$, then
	\begin{equation*}
		(D^2u+u\Rmnum{1})(\cdot,t)\ge \frac{1}{C},\qquad \forall t\in[0,T),
	\end{equation*}
	where $C$ is a positive constant depending only on $\beta,C_0,u(\cdot,0)$, $\vert G\vert_{L^\infty(U)}$, $\vert1/G\vert_{L^\infty(U)}$, $\vert G\vert_{C^1_{x,X}(U)}$ and $\vert G\vert_{C^2_{x,X}(U)}$ where $U=[0,T)\times\mS^n\times B^n_{C_0}$ ($B^n_{C_0}$ is the ball
	centered at the origin with radius $C_0$ in $\mR^n$).
\end{lemma}
\begin{proof}
We again consider hypersurface $M_t$ whose support function is $u(\cdot, t)$. We mention that under this representation we consider that $X$ is independent of $x$. Then the eigenvalues of $[D^2u+u\Rmnum{1}]$ are the principal radii of curvatures of $M_t$.
The principal radii of curvatures of $M_t$ are also the eigenvalues of $\{h_{il}e^{lj}\}$. To derive a positive lower bound of principal curvatures radii, it suffices to prove that the eigenvalues of $\{h^{il}e_{lj}\}$ are bounded from above. For this end, we consider the following quantity
	$$W(x,t)=\max\{\frac{h^{ij}}{u}(x,t)\xi_i\xi_j: e^{ij}(x)\xi_i\xi_j=1\}.$$
	%where$$\Theta(x,t)=max\{h^{ij}(x,t)\zeta_i\zeta_j: e^{ij}(x,t)\zeta_i\zeta_j=1\}.$$
	
	Fix an arbitrary $T'\in(0,T)$ and assume that $W$ attains its maximum on $\mS^n\times[0,T']$ at $(x_0,t_0)$ with $t_0>0$ (otherwise $W$ is bounded by its initial value and we have done). We choose a local orthonormal frame $e_1,\cdots,e_n$ on $\mS^n$ such that at $X(x_0,t_0)$, $\{h_{ij}\}$ is diagonal, and assume without loss of generality that
	$$h^{11}\ge h^{22}\ge\cdots\ge h^{nn}.$$
	
	Take a vector $\xi=(1,0,\cdots,0)$ at $(x_0, t_0)$, and extend it to a parallel vector field in a neighborhood of $x_0$ independent of $t$, still denoted by $\xi$. Set
	$$\widetilde W(x,t)=\frac{\frac{h^{kl}}{u}\xi_k\xi_l}{e^{kl}\xi_k\xi_l},$$
	which is differential, and there holds
	$$\widetilde W(x,t)\le \widetilde W(x_0,t_0)=W(x_0,t_0)=\frac{h^{11}}{u}(x_0,t_0).$$
	
	Note that at the point $(x_0, t_0)$
	\begin{gather*}
		\p_t\widetilde W=\p_t\frac{h^{11}}{u}, \; D_i\widetilde W=D_i\frac{ h^{11}}{u},\; {\mbox{ and }}\; D^2_{ij}\widetilde W=D^2_{ij}\frac{h^{11}}{u}.
	\end{gather*}
	This implies that $\widetilde W$ satisfies the same evolution as $\frac{h^{11}}{u}$ at $(x_0, t_0)$. Therefore we shall just apply the maximum principle to the evolution of $\frac{h^{11}}{u}$ to obtain the upper bound.
	
	We denote the partial derivatives $\frac{\p h^{pq}}{\p h_{kl}}$ and $\frac{\p^2 h^{pq}}{\p h_{rs}\p h_{kl}}$ by $h^{pq}_{kl}$ and $h^{pq}_{kl,rs}$ respectively. From \cite{UJ}, we have
	\begin{gather*}
		h^{pq}_{kl}=-h^{pk}h^{ql},\\
		h^{pq}_{kl,rs}=h^{pr}h^{ks}h^{ql}+h^{pk}h^{qr}h^{ls},\\
		D_jh^{pq}=h^{pq}_{kl}D_jh_{kl},\\
		D^2_{ij}h^{pq}=h^{pq}_{kl}D^2_{ij}h_{kl}+h^{pq}_{kl,rs}D_ih_{rs}D_jh_{kl}.
	\end{gather*}
Let $J=F^\beta$. We differentiate the equation (\ref{2.19}) at $(x_t,t)$ to obtain
	\begin{gather}
		\frac{\p}{\p t}D_ku=(uG)_kJ+uG J_k-\eta u_k,\\
		\frac{\p}{\p t}D^2_{kl}u=(uG)_{kl}J+(uG)_{k}J_l+(uG)_lJ_k+(uG) J_{kl}-\eta u_{kl}.\label{3.10}
	\end{gather}
	Using (\ref{2.19}) and (\ref{3.10}), we see that $h_{kl}=D^2_{kl}u+\delta_{kl}u$ satisfies the equation
	$$\frac{\p}{\p t}h_{kl}=(uG)_{kl}J+(uG)_{k}J_l+(uG)_lJ_k+uG J_{kl}+uG J\delta_{kl}-\eta h_{kl},$$
	thus
	\begin{align}\label{3.11}
		\frac{\p}{\p t}h^{11}=&h^{11}_{kl}\frac{\p}{\p t}h_{kl}\notag\\
		=&-(h^{11})^2((uG)_{11}J+2(uG)_1J_1)\notag\\
		&-(h^{11})^2uG(J^{ij,mn}D_1h_{ij}D_1h_{mn}+J^{ij}D_1D_1h_{ij})-uGJ(h^{11})^2+\eta h^{11}.
	\end{align}
	Note that
	\begin{align}\label{3.12}
D_kD_lh^{11}&=h^{11}_{mn}D_kD_lh_{mn}+h^{11}_{mn,rs}D_lh_{mn}D_kh_{rs}\notag\\
		&=-(h^{11})^2D_kD_lh_{11}+h^{ms}(h^{11})^2D_lh_{1m}D_kh_{1s}+h^{ns}(h^{11})^2D_lh_{1n}D_kh_{1s}\notag\\
		&=-(h^{11})^2D_kD_lh_{11}+2(h^{11})^2h^{rs}D_1h_{lr}D_1h_{ks}.
	\end{align}
	By the Ricci identity,
	\begin{equation}\label{3.13}
D_kD_lh_{11}=D_1D_1h_{kl}+\delta_{1k}h_{1l}-h_{kl}+\delta_{lk}h_{11}-\delta_{1l}h_{1k},
	\end{equation}
Combination of (\ref{3.11}), (\ref{3.12}) and (\ref{3.13}) gives
	\begin{align*}
\frac{\p}{\p t}h^{11}=&uG J^{kl}D_kD_lh^{11}-uG(h^{11})^2(J+J^{kl}h_{kl})+uG\sum_iJ^{ii}h^{11}+\eta h^{11}\\&-uG(h^{11})^2(2J^{km}h^{nl}+J^{kl,mn})D_1h_{kl}D_1h_{mn}\\&-(h^{11})^2(2D_1(uG)D_1J+JD_1D_1(uG)).
	\end{align*}
	Since $F=J^\frac{1}{\beta}$ is homogenous of degree one, and satisfies Assumption \ref{a1.1}, it follows from Urbas \cite{UJ} that,
	\begin{gather}
		F^{ij}h_{ij}=F,\\
		(2F^{km}h^{nl}+F^{kl,mn})D_1h_{kl}D_1h_{mn}\geq2F^{-1}F^{kl}F^{mn}D_1h_{kl}D_1h_{mn}.
	\end{gather}
	We then have
	\begin{align*}
		\frac{\p}{\p t}h^{11}=&\beta uG F^{\beta-1}F^{kl}D_kD_lh^{11}-uG F^\beta(\beta+1)(h^{11})^2+\beta uG F^{\beta-1}\sum_iF^{ii}h^{11}+\eta h^{11}\\
		&-uG(h^{11})^2(2\beta F^{\beta-1}F^{km}h^{nl}+\beta(\beta-1)F^{\beta-2}F^{kl}F^{mn}\\&+\beta F^{\beta-1}F^{kl,mn})D_1h_{kl}D_1h_{mn}
		-(h^{11})^2(2\beta F^{\beta-1}D_1(uG)D_1F+F^\beta D_1D_1(uG))\\
		\leq&\beta uG F^{\beta-1}F^{kl}D_kD_lh^{11}-(\beta+1)uG F^\beta(h^{11})^2+\beta uG F^{\beta-1}\sum_iF^{ii}h^{11}+\eta h^{11}\\
		&-\beta(\beta+1)uG F^{\beta-2}(h^{11})^2(D_1F)^2-2\beta F^{\beta-1}D_1(uG)D_1F(h^{11})^2\\
		&-F^\beta(h^{11})^2D_1D_1(uG).
	\end{align*}
	Since $$-2\beta F^{\beta-1}D_1(uG)D_1F(h^{11})^2\leq uG\beta(\beta+1)F^\beta(\frac{D_1F}{F})^2+\frac{\beta}{\beta+1}F^\beta
	\frac{(D_1(uG))^2}{uG},$$
	we have
	\begin{align}\label{3.16}
		\cL h^{11}\leq&-(\beta+1)uGF^\beta(h^{11})^2+\beta uGF^{\beta-1}\sum_iF^{ii}h^{11}+\eta h^{11}\notag\\&+\frac{\beta}{\beta+1}F^\beta
		\frac{(D_1(uG))^2}{uG}(h^{11})^2
		-F^\beta(h^{11})^2D_1D_1(uG),
	\end{align}
where $\cL=\p_t-\beta uGF^{\beta-1}F^{kl}D_kD_l$. Meanwhile, we have
\begin{equation}\label{3.17}
\begin{split}
\cL u=& u(GJ-\eta)-\beta uGF^{\beta-1}F^{kl}D_kD_lu\\
=&u(GJ-\eta)-\beta uGJ+\beta u^2GF^{\beta-1}\sum_{i}F^{ii}.
\end{split}
\end{equation}
Combination of (\ref{3.16}) and (\ref{3.17}) gives
\begin{equation}\label{3.18}
\begin{split}
\cL\frac{h^{11}}{u}=&\frac{\cL h^{11}}{h^{11}}-\frac{\cL u}{u}\\
\le&-uGJh^{11}\(\beta+1-\frac{\beta}{\beta+1}\frac{(D_1(uG))^2}{(uG)^2}+\frac{D_1D_1(uG)}{uG}\)+C_9\\
\le&-uGJh^{11}\(\beta+1-\frac{\beta}{\beta+1}\frac{(D_{x_1}(uG))^2}{(uG)^2}+\frac{D_{x_1}D_{x_1}(uG)}{uG}\)+C_{10},
\end{split}
\end{equation}
where we have used
\begin{align}
D_i(uG)=&D_{x_i}(uG)+d_X(uG)\(\bar\nn_iX\)=D_{x_i}(uG)+d_X(uG)\(\bar\nn_i(ux+Du)\)\notag\\
=&D_{x_i}(uG)+d_X(uG)\(h_{ij}e_j\),\\
h^{11}\frac{D_ju}{u}=&D_jh^{11}=-{(h^{11})}^2D_jh_{11},\notag\\
D_1D_1(uG)=&D_{x_1}D_{x_1}(uG)+d_X\(D_{x_1}(uG)\)\(h_{11}e_1\)+d_X\((uG)_{x_1}\)\(h_{11}e_1\)\notag\\
&+h_{11}d_X\((uG)_{X_1}\)\(h_{11}e_1\)-d_X(uG)\(h_{11}x\)+d_X(uG)\(D_jh_{11}e_j\)\notag\\
=&D_{x_1}D_{x_1}(uG)+d_X\(D_{x_1}(uG)\)\(h_{11}e_1\)+d_X\((uG)_{x_1}\)\(h_{11}e_1\)\notag\\
&+h_{11}d_X\((uG)_{X_1}\)\(h_{11}e_1\)-d_X(uG)\(h_{11}x\)-d_X(uG)\(h_{11}\frac{Du}{u}\)
\end{align}
in the last inequality. The calculation can be seen in \cite{CLL}.
Since $(D_{x_1}D_{x_1}(Gu)^\frac{1}{\beta+1}+(Gu)^\frac{1}{\beta+1})>0$  and $Gu$ is defined in a compact set by the $C^0$ estimates. Thus
$$\(\beta+1-\frac{\beta}{\beta+1}\frac{(D_{x_1}(uG))^2}{(uG)^2}+\frac{D_{x_1}D_{x_1}(uG)}{uG}\)\ge C_{11}.$$
By (\ref{3.18}) we have
$$\p_t\frac{h^{11}}{u}\leq-C_{11}h^{11}+C_{10}.$$
	That is $h^{11}\leq C$ by maximum principle and we have completed this proof.
\end{proof}

If $F=\sigma_{n}^\frac{1}{n}$, condition (ii) of Theorem \ref{t1.2} is not necessary. The corresponding result is proved in \cite{CL}, we list it here and don't repeat the proof.
\begin{lemma}\cite{CL}\label{lcl}
	Let $\beta>0$. Let $u(\cdot,t)$ be a positive, smooth and uniformly convex solution to
	\begin{equation*}
		\frac{\p u}{\p t}=\Psi(t,x,u,Du){\det}^\frac{\beta}{n}(D^2u+u\Rmnum{1})-\eta(t)u \qquad \text{on }\mS^n\times[0,T),
	\end{equation*}
	where $\Psi(t,x,u,Du): [0, \infty)\times\mS^n\times [0, \infty)\times\mR^n\rightarrow [0, \infty)$ is a smooth function. Suppose that $\Psi(t,x,u,Du)>0$ whenever $u>0$. If
	\begin{align*}
		\frac{1}{C_0}\le u(x,t)\le C_0,\qquad &\forall(x,t)\in\mS^n\times[0,T),\\
		\vert Du\vert(x,t)\le C_0,\qquad &\forall(x,t)\in\mS^n\times[0,T),\\
		0\le\eta(t)\le C_0,\qquad&\forall t\in[0,T),
	\end{align*}
	for some constant $C_0>0$, then
	\begin{equation*}
		(D^2u+u\Rmnum{1})(x,t)\ge C^{-1}\Rmnum{1},\qquad \forall(x,t)\in\mS^n\times[0,T),
	\end{equation*}
	where $C$ is a positive constant depending only on $n,\beta,C_0,u(\cdot,0)$, $\vert\Psi\vert_{L^\infty(U)}$, $\vert1/\Psi\vert_{L^\infty(U)}$, $\vert\Psi\vert_{C^1_{x,u,Du}(U)}$ and $\vert\Psi\vert_{C^2_{x,u,Du}(U)}$ where $U=[0,T)\times\mS^n\times[1/C_0,C_0]\times B^n_{C_0}$ ($B^n_{C_0}$ is the ball
	centered at the origin with radius $R$ in $\mR^n$).
\end{lemma}

\begin{lemma}\label{l3.6}
	Let $\beta>0$. Let $u(\cdot,t)$ be a positive, smooth and uniformly convex solution to
	\begin{equation*}
		\frac{\p u}{\p t}=uG(t,x,X)F^\beta(D^2u+u\Rmnum{1})-\eta(t)u\quad\text{ on }\mS^n\times[0,T),
	\end{equation*}
where $G(t,x,X): [0, \infty)\times\mS^n\times\mR^{n+1}\rightarrow(0,\infty)$ is a smooth function if $\la x,X\ra>0$. If
\begin{equation*}
		\frac{1}{C_0}\le u(x,t)\le C_0,\qquad\qquad \forall(x,t)\in\mS^n\times[0,T).
\end{equation*}
Suppose either $F$ satisfies Assumption \ref{a1.1}, $(D_{x_i}D_{x_j}(Gu)^\frac{1}{\beta+1}+(Gu)^\frac{1}{\beta+1}\delta_{ij})>0$,
\begin{align*}
	\eta(t)\le	C_0,\qquad\qquad&\forall t\in[0,T),\\
	\frac{1}{C_0}\le F(x,t)\le C_0,\qquad\qquad &\forall(x,t)\in\mS^n\times[0,T),
\end{align*}
or $F=\sigma_{n}^\frac{1}{n}$,
\begin{align*}
0\le\eta(t)\le	C_0,\qquad\qquad&\forall t\in[0,T),\\
\sigma_{n}^\frac{1}{n}\le C_0,\qquad\qquad &\forall(x,t)\in\mS^n\times[0,T),
\end{align*}
for some constant $C_0>0$. Then
	\begin{equation*}
	\frac{1}{C}\le	(D^2u+u\Rmnum{1})(\cdot,t)\le C,\qquad \forall t\in[0,T),
	\end{equation*}
where $C$ is a positive constant depending only on $\beta,C_0,u(\cdot,0)$, $\vert G\vert_{L^\infty(U)}$, $\vert1/G\vert_{L^\infty(U)}$, $\vert G\vert_{C^1_{x,X}(U)}$ and $\vert G\vert_{C^2_{x,X}(U)}$ where $U=[0,T)\times\mS^n\times B^n_{C_0}$ ($B^n_{C_0}$ is the ball
	centered at the origin with radius $C_0$ in $\mR^n$).
\end{lemma}
\begin{proof}
By Lemma \ref{l3.5} and \ref{lcl}, we have the left inequality and $\frac{1}{\l_i}\leq C$ for $i=1,\cdots,n$. Since $F_*(\frac{1}{\l_1},\cdots,\frac{1}{\l_n})=\frac{1}{F(\l_1,\cdots,\l_n)}$ is uniformly continuous on $\overline{\Gamma}_{C}=\{\l\in\overline{\Gamma}_+\vert\l_i\geq \frac{1}{C}$ for all $i \}$, and $F_*$ is bounded from below by a positive constant.  Assumption \ref{a1.1} implies that $\l_i$ remains in a fixed compact subset of $\bar\Gamma_+$, which is independent of $t$. That is $$\frac{1}{C}\leq\l_i(\cdot,t)\leq C.$$
\end{proof}

\section{Convergence and proof of theorems}
By Section 3 we only need to establish the $C^0$ estimate to derive the a priori estimates.
\begin{lemma}\label{l4.1}
	Let $u$ be a smooth, strictly convex solution of (\ref{2.19}) on $\mS^n\times[0,T^*)$, and let $G$ satisfy the condition (i) of Theorem \ref{t1.2}. Let $F$ satisfy the condition (i) of Assumption \ref{a3.1}. Then
	$$\frac{1}{C_{12}}\le u(x,t)\le  C_{12},$$
	where $C_{12}$ is a positive constant depending on $\lim\sup_{s\rightarrow+\infty}[\max_{y=sx}G(x,y)s^\beta]$, $\max_{\mS^n}u(x,0)$, $\min_{\mS^n}u(x,0)$, and $\lim\inf_{s\rightarrow0^+}[\min_{y=sx}G(x,y)s^\beta]$.
\end{lemma}
\begin{proof}
	Let $u_{\max}(t)=\max_{x\in \mS^n}u(\cdot,t)=u(x_t,t)$. For fixed time $t$, at the point $x_t$, we have $$D_iu=0 \text{ and } D^2_{ij}u\leq0.$$
	Note that $h_{ij}=D^2_{ij}u+u\delta_{ij}$.
	At the point $x_t$, we have $F^{\beta}(h_{ij})\leq u^{\beta}$. Then
	$$\frac{d}{dt}u_{\max}\leq u(Gu^\beta-1).$$
	Let $\bar{A}=\lim\sup_{s\rightarrow+\infty}[\max_{y=sx}G(x,y)s^\beta]$. Note that $\rho(x_t)=\max_{\mS^n}\rho=\max_{\mS^n}u$. If we denote $X=\rho\xi$, then $u(x_t)=\la X,x\ra=\rho(x_t)\la x,\xi\ra$, i.e. $\la x,\xi\ra|_{x_t}=1$ and $X=u({x_t})x$.
	By the condition (i) of Theorem \ref{t1.2}, $\eps=\frac{1}{2}(1-\bar{A})$ is positive and there exists a positive constant $C_{13}>0$ such that
	$$Gu^\beta<\bar{A}+\eps$$
	for $u>C_{13}$. It follows that
	$$Gu^\beta-1<\bar{A}+\eps-1<0,$$
	from which we have $\p_tu<0$ at maximal points. Hence
	$$u_{\max}\leq \max\{C_{13},u_{\max}(0)\}.$$
	Similarly, at the minimum point,
	$$\frac{d}{dt}u_{\min}\geq u(Gu^\beta-1).$$
	By the condition (i) of Theorem \ref{t1.2}, there exists a positive constant $C_{14}>0$ such that $Gu^\beta-1>0$ if $u<C_{14}$. Thus
	$$u_{\min}\geq \min\{C_{14},u_{\min}(0)\}.$$
\end{proof}

The estimates obtained in Lemma  \ref{l3.3}, \ref{l3.4}, \ref{l3.6} and \ref{l4.1}  are independent of $T^*$. By Lemma \ref{l4.1} and \ref{l3.6}, we conclude that the equation (\ref{2.19}) is uniformly parabolic. By the $C^0$ estimate (Lemma \ref{l4.1}), the gradient estimate, the $C^2$ estimate (Lemma \ref{l3.6}), Cordes and Nirenberg type estimates \cite{B2,LN} and the Krylov's theory \cite{KNV}, we get the H$\ddot{o}$lder continuity of $D^2u$ and $u_t$. Then we can get higher order derivation estimates by the regularity theory of the uniformly parabolic equations. Hence we obtain the long time existence and $C^\infty$-smoothness of solutions for the normalized flow (\ref{2.19}). The uniqueness of smooth solutions also follows from the parabolic theory.

Finally, we complete the proof of Theorem \ref{t1.2}, \ref{t1.3} and \ref{t1.4}, namely we will prove that the support function $u_\infty$ of $M_\infty$ satisfies the following equation
$$GF^\beta=1.$$
To achieve this, we need to find a quantity which can preserve monotonicity along flows.

\begin{proof of theorem 1.2}
By the condition (\rmnum{3}) of Theorem \ref{t1.2}, we choose the initial hypersurface $M_0$ satisfied $G(x,X)F^\beta\vert_{M_0}\ge1$ or $G(x,X)F^\beta\vert_{M_0}\le1$. We shall derive this condition is preserved along flows. Let $Q=G(x,X)F^\beta$, along flow (\ref{2.19}), note that
\begin{equation}
\p_tX=\p_t(ux+Du)=(Q-1)ux+(Q-1)Du+uDQ=(Q-1)X+uDQ.
\end{equation}
We have
\begin{equation}\label{5.1}
\begin{split}
\p_tQ=&\frac{d_XG(X)}{G}(Q-1)Q+\frac{d_XG(uDQ)}{G}Q+\beta\frac{Q}{F}F^{ij}\(((Q-1)u)_{ij}+((Q-1)u)\delta_{ij}\)\\
=&\(\frac{d_XG(X)}{G}+\beta\)(Q-1)Q+\frac{d_XG(uDQ)}{G}Q+\beta\frac{uQ}{F}F^{ij}Q_{ij}-2\beta\frac{Q}{F}F^{ij}Q_iu_j.
\end{split}
\end{equation}
By (\ref{5.1}), we can find condition (\rmnum{3}) is preserved along flow (\ref{1.2}) for all $t\ge0$. Moreover, due to the $C^0$ estimates,
$$\vert u(x,t)-u(x,0)\vert=\vert\int_{0}^t\(GF^\beta-1\)u(x,s)ds\vert<\infty.$$
Therefore, in view of the monotonicity of $u$, the limit $u(x,\infty):=\lim_{t_i\rightarrow\infty}u(x,t)$ exists and is positive, smooth and $D^2u(x,\infty)+u(x,\infty)\Rmnum{1}>0$. Thus the hypersurface with support function $u(x,\infty)$ is our desired solution to $$GF^\beta=1.$$

For the proof of Corollary \ref{c1.2} if $G=\psi(x)u^{\a-1}\rho^{\delta}$, we only need to prove the uniqueness. This is a special case of Theorem \ref{t5.1} in Section 5.
\end{proof of theorem 1.2}

\begin{proof of theorem 1.4}
At this case, $G=\varphi(x,u)$ and $F^k$ is divergence-free. We consider two quantities,
$$W_F=\frac{1}{k+1}\int_{\mS^n}uF^k dx \qquad \text{    and    }\qquad U=\int_{\mS^n}\int_{\eps}^u\varphi^{-\frac{k}{\beta}}(x,s)dsdx,$$
where $\eps$ is a positive constant. We prefer to get some properties of the derivative of $U$, thus $\eps$ is just chosen such that $U$ is well-defined. Let $Q=\varphi F^{\beta}$, we have
\begin{equation}
\begin{split}
\p_tW_F=&\frac{1}{k+1}\int_{\mS^n}uF^k(Q-1)+u(F^k)^{ij}(Q_{ij}u+2Q_iu_j+(Q-1)h_{ij})dx\\
=&\int_{\mS^n}uF^k(Q-1)dx,
\end{split}
\end{equation}
where we have used $F^k$ is homogeneous of degree $k$ and $\sum_{i\ge1}D_i(F^k)^{ij}=0$. Thus,
\begin{equation}\label{4.4}
\begin{split}
\p_t(W_F-U)=&\int_{\mS^n}F^k u(Q-1)-\varphi^{-\frac{k}{\beta}}(x,u)\p_tudx\\
=&\int_{\mS^n}u(\varphi F^\beta-1)\(F^k-\varphi^{-\frac{k}{\beta}}\)dx\\
=&\int_{ \{F^k\varphi^{\frac{k}{\beta}}-1\ne0\}}u\varphi^{-\frac{k}{\beta}}(\varphi F^{\beta}-1)^2\dfrac{F^k\varphi^{\frac{k}{\beta}}-1}
{\varphi F^{\beta}-1}dx+\int_{\{F^k\varphi^{\frac{k}{\beta}}-1=0\}}0dx.
\end{split}
\end{equation}
If $F^k\varphi^{\frac{k}{\beta}}>1$, we have $\varphi F^{\beta}=(F^k\varphi^{\frac{k}{\beta}})^\frac{\beta}{k}>1$ by $\frac{\beta}{k}>0$. If $F^k\varphi^{\frac{k}{\beta}}<1$, we have $\varphi F^{\beta}=(F^k\varphi^{\frac{k}{\beta}})^\frac{\beta}{k}<1$. Thus
$$\dfrac{F^k\varphi^{\frac{k}{\beta}}-1}{\varphi F^{\beta}-1}>0 \quad \text{  if  }\quad F^k\varphi^{\frac{k}{\beta}}-1\ne0.$$
This implies that
$$\p_t(W_F-U)\ge0,$$
equality holds if and only if $u$ solves $\varphi F^{\beta}=1$. Similar to the proof of Theorem \ref{t1.2}, by the bounds of $W_F$ and $U$ and the monotonicity of $W_F-U$, we obtain that
$$\int_{0}^\infty\vert\frac{d}{dt}(W_F-U)\vert dt\le C.$$
Hence there is a sequence of $t_i\rightarrow\infty$ such that $\frac{d}{dt}(W_F-U)\rightarrow0$, i.e. $u(\cdot,t_i)$ converges smoothly to a positive, smooth and uniformly convex function $u_\infty$ solving $\varphi F^{\beta}=1.$
\end{proof of theorem 1.4}

\begin{proof of theorem 1.5}
At this case $G=\varphi(x,u)\phi(X)$, $F=\sigma_{n}^\frac{1}{n}(D^2u+u\Rmnum{1})$. Since $M_t$ is a graph of a smooth and positive function $\rho(\xi)$ on $\mS^n$, we can denote $X=\rho\xi$ and $\phi(X)=\phi(\xi,\rho)$. Similar to the proof of Theorem \ref{t1.3},  We consider two quantities,
$$V=\int_{\mS^n}\int_\eps^\rho\phi^\frac{n}{\beta}(\xi,s)s^ndsd\xi \qquad \text{    and    }\qquad U=\int_{\mS^n}\int_{\eps}^u\varphi^{-\frac{n}{\beta}}(x,s)dsdx,$$
where $\eps$ is a positive constant.  Thus by $K=\frac{1}{\sigma_n(\l)}$, (\ref{2.22}) and (\ref{2.24}),
\begin{equation}\label{4.5}
	\begin{split}
		\p_t(V-U)=&\int_{\mS^n}\phi^\frac{n}{\beta}(\xi,\rho)\rho^n\p_t\rho d\xi-\int_{\mS^n}\varphi^{-\frac{n}{\beta}}(x,u)\p_tudx\\
		=&\int_{\mS^n}\(\phi^\frac{n}{\beta}\sigma_n-\varphi^{-\frac{n}{\beta}}\)\p_tu dx\\
		=&\int_{ \{\sigma_n(\phi\varphi)^{\frac{n}{\beta}}\ne1\}}u\varphi^{-\frac{n}{\beta}}(\varphi\phi\sigma_n^\frac{\beta}{n}-1)^2\dfrac{\sigma_n\phi^\frac{n}{\beta}\varphi^{\frac{n}{\beta}}-1}{\varphi\phi\sigma_n^\frac{\beta}{n}-1}dx+\int_{\{\sigma_n(\phi\varphi)^{\frac{n}{\beta}}=1\}}0dx.
	\end{split}
\end{equation}
By $\varphi\phi\sigma_n^\frac{\beta}{n}=(\sigma_n\phi^\frac{n}{\beta}\varphi^{\frac{n}{\beta}})^\frac{\beta}{n}$ we have
$$\dfrac{\sigma_n\phi^\frac{n}{\beta}\varphi^{\frac{n}{\beta}}-1}{\varphi\phi\sigma_n^\frac{\beta}{n}-1}>0 \quad \text{  if  }\quad \sigma\varphi^{\frac{m}{\beta}}-1\ne0.$$
Thus, it follows from (\ref{4.5}) that
$$\p_t(V-U)\ge0,$$
equality holds if and only if $u$ solves $\varphi\phi\sigma_n^\frac{\beta}{n}=1$. Similar to the proof of Theorem \ref{t1.2}, by the bounds of $V$ and $U$ and the monotonicity of $V-U$, we obtain that
$$\int_{0}^\infty\vert\frac{d}{dt}(V-U)\vert dt\le C.$$
Hence there is a sequence of $t_i\rightarrow\infty$ such that $\frac{d}{dt}(V-U)\rightarrow0$, i.e. $u(\cdot,t_i)$ converges smoothly to a positive, smooth and uniformly convex function $u_\infty$ solving $\varphi\phi\sigma_n^\frac{\beta}{n}=1.$
\end{proof of theorem 1.5}

\section{The uniqueness results}
In this section, we shall derive two uniqueness results in some particular cases. We firstly state a uniqueness result which is proved by an expanding flow.
\begin{theorem}\label{t1.6}
	Let $G=G(u,\rho)$, $F\in C^2(\Gamma_+)\cap C^0(\p\Gamma_+)$ satisfy Assumption \ref{a1.1}, and let $M_0$ be a closed, smooth, uniformly convex hypersurface in $\mathbb{R}^{n+1}$, $n\ge2$, enclosing the origin. Suppose
	
	$(i)$ $\lim_{s\rightarrow+\infty}[G(s,s)s^\beta]<1<\lim_{s\rightarrow0^+}[G(s,s)s^\beta]$,
	
	$(ii)$ $(D_{x_i}D_{x_j}(Gu)^\frac{1}{\beta+1}+(Gu)^\frac{1}{\beta+1}\delta_{ij})>0$,
	
	$(iii)$ $\frac{uG_u}{G}+\frac{\rho G_\rho}{G}+\beta\le0.$
	
	\noindent Then flow (\ref{1.2}) has a unique smooth strictly convex solution $M_t$ for all time $t>0$. And
	$M_t$ converges exponentially to a round sphere centered at the origin in the $C^\infty$-topology.
	
	This means that the solutions of $G(u,\rho)F^\beta=1$ must be a sphere whose radius $R$ satisfies $G(R,R)R^\beta=1$.
	
	In particular, if $F=\sigma_{n}(\l)^\frac{1}{n}$, the condition (ii) can be dropped.
\end{theorem}

\textbf{Remark}: If $G=u^{\a-1}\rho^\delta$, we prove the long time existence and  convergence of this flow with $\a+\delta+\beta<1$, $\a\le0$. Meanwhile, we can prove the similar result with $\a+\delta+\beta=1$, $\a\le0$ by the same method.

\begin{proof}
It suffices to prove the convergence by the proof of Theorem \ref{t1.2}. Let $\omega=\log u$. Then we have
\begin{gather*}
	\omega_i=\frac{u_i}{u},\\\omega_{ij}=\frac{u_{ij}}{u}-\frac{u_iu_j}{u^2},\\
	h_{ij}=u_{ij}+u\delta_{ij}=u(\omega_{ij}+\omega_i\omega_j+\delta_{ij}).
\end{gather*}
It is easy to see that $\omega_{ij}$ is symmetric. Thus
$$\p_t\omega=\frac{\p_tu}{u}=G(u,\rho)u^{\beta}F^\beta([\omega_{ij}+\omega_i\omega_j+\delta_{ij}])-1.$$
Consider the auxiliary function $Q=\frac{1}{2}\vert D\omega\vert^2$. Note that $F=F([\omega_{ij}+\omega_i\omega_j+\delta_{ij}])$ in the rest proof. At the point where $Q$ attains its spatial maximum, we have
\begin{gather*}
	0=D_iQ=\sum_{l}\omega_{li}\omega_{l},\\
	0\geq D_{ij}^2Q=\sum_{l}\omega_{li}\omega_{lj}+\sum_{l}\omega_{l}\omega_{lij}.
\end{gather*}
Note that $D\rho=D(u\sqrt{1+|D\om|^2})=\sqrt{1+|D\om|^2}Du=\rho D\om$. We have
\begin{align*}
	\dfrac{\p_tQ_{\max}}{\om_t+1}&=\dfrac{1}{\om_t+1}\sum\omega_l\omega_{lt}\\
	&=\omega_l\((\frac{uG_u}{G}+\frac{\rho G_\rho}{G}+\beta)\omega_l+\beta \frac{F^{ij}}{F}(\omega_{ijl}+\omega_{il}\omega_j+\omega_i\omega_{jl})\)\\
	&=\(\frac{uG_u}{G}+\frac{\rho G_\rho}{G}+\beta\)|D\om|^2+\beta\frac{F^{ij}}{F}\omega_l\omega_{ijl}.
\end{align*}
By the Ricci identity,$$D_l\omega_{ij}=D_j\omega_{li}+\delta_{il}\omega_j-\delta_{ij}\omega_l,$$
we get
\begin{align}\label{5}
	\dfrac{\p_tQ_{\max}}{\om_t+1}&=\(\frac{uG_u}{G}+\frac{\rho G_\rho}{G}+\beta\)|D\om|^2+\beta\frac{F^{ij}}{F}(\omega_l\omega_{lij}+\omega_i\omega_j-\delta_{ij}\vert D\omega\vert^2)\notag\\
	&\leq\frac{2\beta}{F}(\max_iF^{ii}-\sum_iF^{ii})Q_{\max}
\end{align}
by condition (iii) and the positive definiteness of the symmetric matrix $[F^{ij}]$. Since we have proved the a priori estimates, we have the bound of $\om_t+1$, $F^{ii}$ and $F$, i.e. $\max_iF^{ii}-\sum_iF^{ii}\le -C_0$. This proves
\begin{equation*}
	\max_{\mS^n}\frac{\vert Du(\cdot,t)\vert}{u(\cdot,t)}\leq Ce^{-C_0t},\forall t>0,
\end{equation*}
for both $C$ and $C_0$ are positive constants. In other words, we have that $|D u|\to0$ exponentially as $t\to\infty$. Hence by the interpolation and the a priori estimates, we can get that $u$ converges exponentially to a constant in the $C^\infty$ topology as $t\to\infty$. By Theorem \ref{t1.2}, there exists a solution of $G(u,\rho)F^\beta=1$. Since this solution is invariant along flow (\ref{1.2}), the solution must be a sphere whose radius $R$ satisfies $G(R,R)R^\beta=1$.
\end{proof}

Finally we provide another uniqueness result.  The following proof is inspired by \cite{CW,DL3}.
\begin{theorem}\label{t5.1}
Assume $G(x,X)=G(x,u,\rho)$ and $F$ satisfies the condition (\rmnum{1}) of Assumption \ref{a3.1}. If whenever
\begin{equation}\label{5.6}
G(x,ms_1,ms_2)\ge G(x,s_1,s_2)m^{-\beta}
\end{equation}
holds for some positive $s_1$, $s_2$ and $m$, there must be $m\le1$. Then the solution to the equation
\begin{equation}\label{5.7}
G(x,u,\rho)F^\beta=c
\end{equation}
is unique, where $c$ is an arbitrary constant.
\end{theorem}
\begin{proof}
Let $u_1$, $u_2$ be two smooth solutions of (\ref{5.7}), i.e.
\begin{equation*}
G(x,u_1,\rho_1)F^{\beta}(D^2u_1+u_1\Rmnum{1})=c,\qquad G(x,u_2,\rho_2)F^{\beta}(D^2u_2+u_2\Rmnum{1})=c.
\end{equation*}
Suppose $M=\frac{u_1}{u_2}$ attains its maximum at $x_0\in\mS^n$, then at $x_0$,
\begin{align*}
	0=D\log M=&\frac{Du_1}{u_1}-\frac{Du_2}{u_2},\\
	0\ge D^2\log M=&\frac{D^2u_1}{u_1}-\frac{D^2u_2}{u_2}.
\end{align*}
Hence at $x_0$, by $\rho=u\sqrt{1+\vert D\log u\vert^2}$ we get
\begin{equation*}
	\begin{split}
1=&\frac{G(x,u_1,\rho_1)F^{\beta}(D^2u_1+u_1\Rmnum{1})}{G(x,u_2,\rho_2)F^{\beta}(D^2u_2+u_2\Rmnum{1})}\\
=&\frac{G(x,u_1,u_1\sqrt{1+\vert D\log u_1\vert^2})u_1^{\beta}F^{\beta}(\frac{D^2u_1}{u_1}+\Rmnum{1})}{G(x,u_2,u_2\sqrt{1+\vert D\log u_2\vert^2})u_2^{\beta}F^{\beta}(\frac{D^2u_2}{u_2}+\Rmnum{1})}\\
\le&\frac{G(x,u_2M,u_2M\sqrt{1+\vert D\log u_1\vert^2}) M^\beta}{G(x,u_2,u_2\sqrt{1+\vert D\log u_1\vert^2})}.
	\end{split}
\end{equation*}
Let $m=M$, $s_1=u_2$, $s_2=u_2\sqrt{1+\vert D\log u_1\vert^2}$, we have $M\le1$ by our assumption (\ref{5.6}).

 Similarly one can show $\min_{\mS^n}M\ge1$. Therefore $u_1\equiv u_2$.
\end{proof}
\textbf{Remark:} If $G=\psi(x)u^{\a-1}\rho^\delta$, we prove a uniqueness result with $\a+\delta+\beta<1$ in Theorem \ref{t5.1}.

\section{Orlicz Christoffel-Minkowski type problem}
In this section, we shall prove Theorem \ref{t1.7} and give some inequalities involving the modified quermassintegrals. Firstly, we have the following essential lemma. Due to the lack of symmetric of complete polarization of $W_F(K)$ and suitable geometric inequality (mixed quermassintegral type inequality), we need to consider the variational formula via flow's view.
\begin{lemma}\label{l6.1}
Let $F,K,L$ satisfy the assumptions in Definition \ref{d1.6}. Suppose that
\begin{equation*}
\lim_{\eps\rightarrow0^+}\frac{u_{K+_{\varphi,\eps}L}(x)-u_K(x)}{\eps}=J_{K,L,\varphi_1,\varphi_2}(x),
\end{equation*}
uniformly for $x\in\mS^n$, where $J_{K,L,\varphi_1,\varphi_2}:\mS^n\rightarrow[0,\infty)$ is a measurable function. Then
\begin{equation}
\lim_{\eps\rightarrow0^+}\frac{W_F(K+_{\varphi,\eps}L)-W_F(K)}{\eps}=\int_{\mS^n}J_{K,L,\varphi_1,\varphi_2}(x)F^k(\l_{\p K})dx.
\end{equation}
\end{lemma}
\begin{proof}
Let $K_\eps=K+_{\varphi,\eps}L$. Since $K$ and $L$ are smooth domains, we have
\begin{align*}
&\lim_{\eps\rightarrow0^+}\frac{W_F(K+_{\varphi,\eps}L)-W_F(K)}{\eps}=\frac{d}{d\eps}W_F(K_\eps)|_{\eps=0}\\
=&\frac{1}{k+1}\frac{d}{d\eps}\int_{\mS^n}u_{K_\eps}F^k([D^2_{ij}u_{K_\eps}+u_{K_\eps}\delta_{ij}])dx|_{\eps=0}\\
=&\frac{1}{k+1}\int_{\mS^n}J_{K,L,\varphi_1,\varphi_2}(x)F^k(\l_{\p K})+u_K(F^k)^{ij}\(D^2_{ij}J_{K,L,\varphi_1,\varphi_2}+J_{K,L,\varphi_1,\varphi_2}\delta_{ij}\)dx\\
=&\frac{1}{k+1}\int_{\mS^n}J_{K,L,\varphi_1,\varphi_2}(x)F^k(\l_{\p K})+(F^k)^{ij}\(D^2_{ij}u_KJ_{K,L,\varphi_1,\varphi_2}+J_{K,L,\varphi_1,\varphi_2}u_K\delta_{ij}\)dx\\
=&\int_{\mS^n}J_{K,L,\varphi_1,\varphi_2}(x)F^k(\l_{\p K})dx,
\end{align*}
where we have used that $F^k$ is divergence-free and homogeneous of degree $k$.
\end{proof}
The following variation of $u$ has been derived in \cite{GHW}.
\begin{lemma}\cite{GHW}\label{l6.2}
Let $K\in\cK_{oo}^{n+1}$ and $L\in\cK_{o}^{n+1}$. Then
\begin{equation*}
\lim_{\eps\rightarrow0^+}\frac{u_{K+_{\varphi,\eps}L}(x)-u_K(x)}{\eps}=\frac{u_K(x)}{(\varphi_1)'_l(1)}\varphi_2\(\frac{u_L(x)}{u_K(x)}\),
\end{equation*}
uniformly for $x\in\mS^n$.
\end{lemma}
\noindent\begin{proof of theorem 1.7}
Combining Lemma \ref{l6.1} and \ref{l6.2}, we complete the proof of Theorem \ref{t1.7}.
\end{proof of theorem 1.7}

We have studied the question of evolving a smooth closed convex hypersurface in $\mR^{n+1}$ into another such hypersurface. A natural idea to prove Corollary \ref{t1.8} is using the monotone quantities obtained along our flows. We remark that more meaningful inequalities obtained by flow method will be given in \cite{DLY}.

\noindent\begin{proof of corollary 1.8}
We only prove Corollary \ref{t1.8}(1) since the proof of Corollary \ref{t1.8}(2) is similar. Let $M_0=\p K$. Along the flow in Theorem \ref{t1.3}, one of the convergent hypersurfaces is $\p L$. Due to the monotonicity of $W_F-U$ (\ref{4.4}) in the proof of Theorem \ref{t1.3}, this proof has been completed.
\end{proof of corollary 1.8}

To make it easier to understand these two inequalities, we shall give some remarks and corollaries on Corollary \ref{t1.8}. Firstly, we look at the homogeneous case of Corollary \ref{t1.8}(1), i.e. $\varphi^{-\frac{k}{\beta}}(x,u)=f(x)u^{p-1}$. We shall recall some basic definitions. For details beyond the short treatment we present below, the reader is referred to \cite{LE,SR}. For $K\in\cK_{o}^{n+1}$, let $W_k(K)=\frac{1}{k+1}\int_{\mS^n}u\sigma_{k}(\l)dx$ denote the $k$th-Quermassintegrals of $K$. Then the $L^p$ mixed quermassintegral $W_{p,k}(K,L)$ can be defined by the first variation of the ordinary quermassintegrals with respect to the $L^p$ linear combination,
\begin{equation*}
W_{p,k}(K,L)=\lim_{\eps\rightarrow0^+}\frac{W_k(K+_{p,\eps} L)-W_k(K)}{\eps}.
\end{equation*}
Thus the $L^p$ mixed Quermassintegral $W_{p,k}(K,L)$ has the following integral
representation:
\begin{equation*}
W_{p,k}(K,L)=\int_{\mS^n}u_L^{p}u_K^{1-p}\sigma_{k}(\l_{\p K})dx.
\end{equation*}
The well-known $L^p$ mixed quermassintegral inequalities are as following,
\begin{equation}
W_{p,k}(K,L)^{n+1-k}\ge W_k(K)^{n+1-k-p}W_k(L)^p,
\end{equation}
with equality if and only if $K$ and $L$ are dilates, which are homogeneous inequalities. As a counterpart, a class of non-homogeneous inequalities will be given in the following corollary. We denote $W_{\varphi,F}(K,L)$ by $W_{p,F}(K,L)$ if $\varphi=u^p$. It is obvious that $W_{p,F}(K,L)$ and $W_{p,k}(K,L)$ differ only by $C_n^k$ multiple if $F^k=\sigma_{k}/C_n^k$.
\begin{corollary}\label{c6.5}
Let $F\in C^2(\Gamma_+)\cap C^0(\p\Gamma_+)$ satisfy Assumption \ref{a1.1}, $F^k$ is divergence-free, $K,L\in\cK_{oo}^{n+1}$ and $p>k+1$.
If $[D^2(u_L^{p-1}\sigma_{k}^{-1}(\l_{\p L}))^\frac{1}{k+p-1}+(u_L^{p-1}\sigma_{k}^{-1}(\l_{\p L}))^\frac{1}{k+p-1}\Rmnum{1}]$ is positive definite, we have
\begin{equation*}
W_{p,F}(L,K)\ge pW_F(K)+(k+1-p)W_F(L),
\end{equation*}
with equality if and only if $K=L$.
\end{corollary}
\noindent\textbf{Remark:} If $F^k=\sigma_{k}/C_n^k$, this corollary is corresponding to the $L^p$ Christoffel-Minkowski setting. We shall let $u_L^{1-p}\sigma_{k}(\l_{\p L})$ be the density function $f(x)$ thus $[D^2(u_L^{p-1}\sigma_{k}^{-1}(\l_{\p L}))^\frac{1}{k+p-1}+(u_L^{p-1}\sigma_{k}^{-1}(\l_{\p L}))^\frac{1}{k+p-1}\Rmnum{1}]$ is required to be positive definite, which is the sufficient condition of the $L^p$ Christoffel-Minkowski problem. If the necessary condition of $L^p$ Christoffel-Minkowski problem could be solved, the positive definite condition may be dropped.
\begin{proof}
Let $\varphi^{-\frac{k}{\beta}}(x,s)=f(x)s^{p-1}$ in Corollary \ref{t1.8}(1). Let $f(x)=u_L^{1-p}F^{k}(\l_{\p L})$. Then $p>k+1$ by assumptions. By Corollary \ref{t1.8} and direct calculation, we have
\begin{align*}
W_F(K)-W_F(L)\le&\int_{\mS^n}\int_{u_L}^{u_K}f(x)s^{p-1}dsdx\\
=&\frac{1}{p}\int_{\mS^n}f(x)u_K^p-f(x)u_L^pdx\\
=&\frac{1}{p}\int_{\mS^n}u_K^pu_L^{1-p}F^{k}(\l_{\p L})dx-\frac{1}{p}\int_{\mS^n}u_LF^{k}(\l_{\p L})dx\\
=&\frac{1}{p}W_{p,F}(L,K)-\frac{k+1}{p}W_F(L).
\end{align*}
By the unique result derived in Theorem \ref{t5.1}, the proof has been completed.
\end{proof}
In particular, if $L$ is a unit ball, we have the following corollary.
\begin{corollary}
	Let $F\in C^2(\Gamma_+)\cap C^0(\p\Gamma_+)$ satisfy Assumption \ref{a1.1}, $F^k$ is divergence-free, $K\in\cK_{oo}^{n+1}$ and $p\ge k+1$. Then
	\begin{equation*}
W_F(K)\le \frac{1}{p}\int_{\mS^n}u^pdx+\(\frac{1}{k+1}-\frac{1}{p}\)|\mS^n|,
	\end{equation*}
	with equality if and only if $K$ is a unit ball.
\end{corollary}
\noindent\textbf{Remark:} If $p=k+1$, the result of flow we need to use has been proved widely in \cite{DL,DL3,IM} etc..

Finally, we consider the homogeneous case of Corollary \ref{t1.8}(2), i.e. $\varphi^{-\frac{n}{\beta}}(x,u)=f(x)u^{p-1}$, $\phi^\frac{n}{\beta}(\xi,s)s^n=s^{q-1}$. We need to introduce the $L^p$ dual curvature measure posed in \cite{LYZ2}. Denote the set of star
bodies in $\mR^{n+1}$ by $\cS_o^{n+1}$. For $p,q\in\mR$, a convex body $K\in\cK_{oo}^{n+1}$, and a star body $Q\in\cS_o^{n+1}$, the $L^p$ dual curvature measures $\widetilde{C}_{p,q}(K,Q,\cdot)$ on $\mS^n$ is defined by
\begin{equation}
d\widetilde{C}_{p,q}(K,Q,\cdot)=u_K^{-p}\rho_K^q\rho_Q^{n+1-q}d\xi.
\end{equation}
The $L^p$ surface area measures, the dual curvature measures and the $L^p$ integral curvatures are special cases, of the $L^p$ dual curvature measures in the sense that for $p,q\in\mR$ and $K\in\cK_{oo}^{n+1}$,
\begin{align*}
\widetilde{C}_{p,n+1}(K,B^{n+1},\cdot)&=S_p(K,\cdot),\\
\widetilde{C}_{0,q}(K,B^{n+1},\cdot)&=\widetilde{C}_{q}(K,\cdot),\\
\widetilde{C}_{p,0}(K,B^{n+1},\cdot)&=J_p(K^*,\cdot),
\end{align*}
where $B^{n+1}$ is the unit ball in $\mR^{n+1}$. Using $L^p$ dual curvature measures, the $L^p$ dual mixed volume $\widetilde{V}_{p,q}(K,L,Q)$ is defined by
\begin{equation*}
\widetilde{V}_{p,q}(K,L,Q)=\int_{\mS^n}u_L^pd\widetilde{C}_{p,q}(K,Q,\cdot)=\int_{\mS^n}u_L^pu_K^{-p}\rho_K^q\rho_Q^{n+1-q}d\xi.
\end{equation*}
Thus an inequality for $L^p$ dual mixed volume has been given in \cite{LYZ2}, which is a generalization of the $L^p$ Minkowski inequality for $L^p$ mixed volume. Suppose $1\le\frac{q}{n+1}\le p$, If $K,L\in\cK_{oo}^{n+1}$ and $Q\in\cS_o^{n+1}$,  Lutwak-Yang-Zhang in \cite{LYZ2} obtained
\begin{equation}\label{6.4}
\widetilde{V}_{p,q}(K,L,Q)^{n+1}\ge V(K)^{q-p}V(L)^pV(Q)^{n+1-q}.
\end{equation}
As a counterpart, a class of non-homogeneous inequalities will be given in the following corollary. Recall the $q$-th dual volume $\widetilde{V}_q(K)$ is defined by $\widetilde{V}_q(K)=\frac{1}{q}\int_{\mS^n}\rho_K^q(\xi)d\xi$ for $q\ne0$, and $\widetilde{V}_q(K)=\int_{\mS^n}\log\rho_K(\xi)d\xi$ for $q=0$. It is obvious that the $(n+1)$-th dual volume $\widetilde{V}_{n+1}(K)=V(K)$. We denote the Aleksandrov’s integral curvature $J_0(K,\cdot)$ by $J(K,\cdot)$.
\begin{corollary}
If $K,L\in\cK_{oo}^{n+1}$ and $p>q$, then
\begin{align}
\widetilde{V}_{p,q}(K,L,B^{n+1})&\ge {p}\widetilde{V}_q(L)+(q-p)\widetilde{V}_q(K), \qquad pq\ne0,\label{6.5}\\
\widetilde{V}_{p,q}(K,L,B^{n+1})&\ge|\mS^n|+p\widetilde{V}_q(L)-p\widetilde{V}_q(K),\quad\; p\ne0,q=0,\\
\int_{\mS^n}\log\frac{u_L}{u_K}d\widetilde{C}_{q}(K,\cdot)&\ge \widetilde{V}_q(L)-\widetilde{V}_q(K), \qquad\qquad\quad p=0,q\ne0,
\end{align}
with any equality if and only if $K=L$.
\end{corollary}
\noindent\textbf{Remark:} Compared with (\ref{6.4}), we constrain $Q=B^{n+1}$ to derive a wide range of $p,q$. In particular, if $q=n+1$, we can derive a $L^p$ Minkowski type inequality by (\ref{6.5}),
\begin{equation}\label{6.8}
V_p(K,L)\ge pV(L)+(n+1-p)V(K), \qquad p> n+1.
\end{equation}
We mention that (\ref{6.8}) is sharper than the $L^p$ Minkowski type inequality if $p< n+1$ by Young's inequality.
\begin{proof}
Let $\varphi^{-\frac{k}{\beta}}(x,s)=f(x)s^{p-1}$ and $\phi^\frac{n}{\beta}(\xi,s)s^n=s^{q-1}$ in Corollary \ref{t1.8}(2). Let $f(x)=u_L^{1-p}\rho_L^{q-n-1}\sigma_{n}(\l_L)$. Then by (\ref{2.27}), (\ref{2.24}) and Corollary \ref{t1.8}, similar to proof of Corollary \ref{c6.5}, the proof has been completed.
\end{proof}

%\reference
\section{Reference}
%\begin{bibdiv}

%\end{bibdiv}

\end{document}